\documentclass[leqno]{article}
\usepackage[utf8]{inputenc}

\usepackage{tikz-cd}

\usepackage{amsmath,amsthm,amscd,amssymb,amsbsy,amstext,mathtools}
\usepackage{pdfpages}
\usepackage{enumerate}
\usepackage{mathrsfs}
\usepackage[utf8]{inputenc}
\usepackage{geometry}[margin=1in]
\numberwithin{equation}{section}
\usepackage{amsfonts}
\usepackage[T1]{fontenc}

\usepackage{hyperref}
\usepackage{url}

\usepackage{authblk}

\newcommand{\eqindent}{\displayindent0pt\displaywidth\textwidth}

\newtheorem{theorem}{Theorem}[section]
\newtheorem{lemma}[theorem]{Lemma}
\newtheorem{proposition}[theorem]{Proposition}
\newtheorem{corollary}[theorem]{Corollary}

\newtheorem{problem}{Problem}

\theoremstyle{definition}

\theoremstyle{remark}
\newtheorem{remark}[theorem]{Remark}

\newcommand{\norm}[1]{\| {#1}\| }

\newcommand{\set}[1]{\left\{#1\right\}}
\newcommand{\void}{\varnothing}
\newcommand{\abs}[1]{\left|#1\right|}

\newcommand{\Rn}{\mathbb{R}^n}
\newcommand{\C}{\mathbb{C}}

\newcommand{\R}{\mathbb{R}}
\newcommand{\brac}[1]{\left(#1\right)}
\newcommand{\jbrac}[1]{\langle#1\rangle}
\newcommand{\ip}[1]{\jbrac{#1}}

\renewcommand{\epsilon}{\varepsilon}
\newcommand{\eps}{\epsilon}

\definecolor{amaranth}{rgb}{0.9, 0.17, 0.31}

\newcommand{\rank}{\mathrm{rk}}

\newcommand{\g}{\gamma}
\newcommand{\gp}{\g_+}
\newcommand{\V}{\mathbb{V}}
\newcommand{\sym}{\mathcal{S}}
\newcommand{\psd}{\mathcal{S}_+}
\newcommand{\tr}{\mathrm{tr}}
\newcommand{\rk}{\mathrm{rk}}
\newcommand{\Sone}{\Sigma_1^n}
\renewcommand{\L}{\mathcal{B}}
\newcommand{\M}{\mathcal{M}}
\newcommand{\kmax}{{k_{\max}}}
\renewcommand{\succ}{\succcurlyeq}

\title{Factorization of positive-semidefinite operators with absolutely summable entries}

\author[1]{Radu Balan}
\author[2]{Fushuai Jiang}

\affil[1]{Department of Mathematics, University of Maryland - College Park}
\affil[2]{Department of Mathematics, City University of Hong Kong}

\date{{\small 
MSC 2020: 15A60, 47B10, 42C15 (Secondary)\\
Keywords: Feichtinger algebra, nuclear operators, 2-summing operators, positive-semidefinite matrices
}
}

\begin{document}
\maketitle

\begin{abstract}
    A problem by Feichtinger, Heil, and Larson asks whether every infinite matrix $A$ with $\sum_{k,l}|A_{kl}| < \infty$ (an equivalent substitute for the Feichtinger algebra) that is positive-semidefinite 
    %and self-adjoint 
    admits a symmetric rank-one decomposition $A = \sum_k f_k^*\otimes f_k$ with $\sum_k \|f_k\|_{1}^2 < \infty$. 
    In the finite-dimensional setting, we analyze the corresponding quantitative $\ell_1^n$ optimization problem by an exact reformulation as a linear program over measures, derive its dual, and prove strong duality. We then obtain an equivalent adjoint formulation regarding the quality of a convex relaxation. 
    In the infinite-dimensional setting, we first provide a negative answer to this question using a concurrent finite-dimensional result by Bandeira-Mixon-Steinerberger. We further study the collection of operators for which such decomposition exists, showing that they are dense in suitable topology and invariant under the action of the positive-coefficient analytic Wiener subalgebra. In addition, we give a sufficient condition for successful rank-one decomposition in terms of $2$-summing factorization, and we characterize exactly when $A^{1/2}$ is $2$-summing.
    
\end{abstract}

\section{Introduction}

In this paper, we study the following problem, first posed by Hans Feichtinger at a 2004 Oberwolfach meeting, later modified by Heil and Larson \cite{HeilLarson08}, and further reformulated by the first author and his collaborators \cite{BOP18, BORWZ}. 

\begin{problem}[Feichtinger-Heil-Larson]\label{prob:Feichtinger}
    Let $A: \ell_2 \to \ell_2$ be a Hermitian positive semidefinite operator with
    \begin{equation}
        \norm{A}_{1,1}:= \sum\limits_{k,l = 1}^{\infty}\abs{\jbrac{A\delta_k, \delta_l} }< \infty.
        \label{eq:condition-nuclear-intro}
    \end{equation}
    Can we always find a sequence $(f_k) \subset \ell_1$ such that 
    \begin{equation}
        A = \sum\limits_{k=1}^{\infty} f_k^*\otimes f_k
    \quad\text{and}\quad
    \sum\limits_{k=1}^{\infty}\norm{f_k}_1^2 < \infty \,?
    \label{eq:condition-strongl2-intro}
    \end{equation}
\end{problem}

We call any $(f_k)_{k=1}^\infty \subset \ell_1$ satisfying the summability condition in \eqref{eq:condition-strongl2-intro} a strong $\ell_2(\ell_1)$ sequence. 

The original problem by Feichtinger referred specifically to eigendecomposition. It was shown in \cite{BOP18} that the eigendecomposition may not provide a strong $\ell_2(\ell_1)$ sequence even though each individual eigenvector belongs to $\ell_1$  (as shown in \cite{HeilLarson08}). Based on a recent construction in \cite{BMS24}, this paper constructs an example that answers negatively to this problem. To do so, we first need to explore the finite-dimensional setting.

Let $\sym^n$ denote the vector space of self-adjoint $n\times n$ matrices. For $A \in\sym^n$, we define
\begin{equation*}
    \norm{A}_{1,1} = \sum_{k,l = 1}^n \abs{A(k,l)}.
\end{equation*}
Let $\psd^n\subset \sym^n$ denote the cone of positive semidefinite matrices. For $A \in \psd^n$, define
\begin{equation}\label{eq:gamma+}
    \gp(A) := \inf \set{\sum_{k = 1}^N \norm{x_k}_{1}^2 : A = \sum_{k = 1}^N x_kx_k^*}.
\end{equation}

The finite-dimensional (or quantitative) variant of Problem \ref{prob:Feichtinger} is as follows.

\begin{problem}\label{prob:matrix}
    For $n \in \mathbb{N}$, let 
    \begin{equation*}
        C_n:= \sup_{A \in \psd^n\setminus\set{0}}\frac{\gp(A)}{\norm{A}_{1,1}}.
    \end{equation*}
    Is there a universal constant $C_0 \geq 1$ such that $C_n\leq C_0$ for all $n$?
\end{problem}

\subsection{Main results and the structure of this paper}

In the finite-dimensional setting, our main contribution is three-fold.

In Theorem \ref{thm:gamma+superresolution}, we establish that (\ref{eq:gamma+}) -- which is a non-convex optimization problem -- is equivalent to the infinite-dimensional linear program:
\begin{equation}
\eqindent
    \gamma_+(A) = \inf\set{ \int_{\Sone}d\mu : A = \int_{\Sone}xx^*d\mu(x),\, \mu \in \M_+(\Sone) }
    \label{eq:intro-gamma+-measure}
\end{equation} 
where the infimum ranges over all Borel measures supported on the $\ell_1$-unit sphere $\Sone = \set{\|x\|_1 = 1}$ in $\R^n$ or $\C^n$. 

In Theorem \ref{thm:strong duality}, we establish the formal dual of the optimization problem in \eqref{eq:intro-gamma+-measure} show that there is {\em no duality gap}, even though we are working with an infinite-dimensional vector space of measures. More concretely, an alternative formula for $\gp$ is given by
    \begin{equation}
    \eqindent
        \gp(A) = \sup\set{\tr(AT) : T \in \sym^n \text{ and } \jbrac{Tx,x} \leq 1 \,,\forall x \in \Sone }
        \label{eq:intro-gamma+-dual}
    \end{equation}
This paves the path for the following result.

In Theorem \ref{thm:equivalence:pi-gamma} (and Theorem \ref{thm:5.1}), we use trace duality to establish an adjoint formulation for Problem \ref{prob:matrix}. In particular,
    \begin{equation}
    \label{eq:intro-dualconvexrelax}
    \eqindent
        \frac{1}{C_n}\max_{A \in \psd^n, \,\norm{A}_{1,1} \leq 1}  \tr(AT) \leq \max_{\norm{x}_{1}^2\leq 1} \jbrac{Tx,x} \leq \max_{A \in \psd^n, \,\norm{A}_{1,1} \leq 1} \tr(AT)
    \end{equation}
    where the optimal $C_n \geq 1$ is the same as the one in Problem \ref{prob:matrix}.
    Note that the optimization problem in the middle is of main interest. It can be written as $\tr[(xx^*)T]$, and the feasible region only consists of truncated rank-one positive-semidefinite rays $xx^*$. On the other hand, the alternative optimization problem ranges over the convex region $\norm{A}_{1,1} \leq 1$ in the cone $A \succ 0$. In Theorem \ref{thm:pi+dual}, we also give an explicit formula for solving this convex optimization problem. 
    See Section \ref{sect:psd-relaxation} for further discussion on the convex relaxation.

The adjoint aforementioned formulation was first announced in the Spring Southeastern AMS Section Meeting 2024 \cite{BalanAMS24}. In \cite{BMS24}, the authors showed that a universal constant $C_0$ in Problem \ref{prob:matrix} cannot exist. In fact, 
\begin{equation}\label{eq: Cn BMS}
    C_n \geq c\sqrt{n}
\end{equation}
where $c > 0$ is some absolute constant. Their construction is done in the adjoint. 
% Thereupon, we show in Theorem \ref{thm:Feichtinger->matrix} that the answer to Problem \ref{prob:Feichtinger} is also no. 

On the other hand, it is known in \cite{BORWZ} that for certain classes of matrices -- including rank-one, diagonally dominant, or $2\times2$ matrices -- the constant $C_n$ may taken to be exactly $1$. In Proposition \ref{prop:special-case-CP}, we expand these known cases to include completely positive matrices. These are the matrices $A \in \psd^n$ that admit a symmetric nonnegative matrix factorization, i.e., $A = XX^*$ with $X_{ij} \geq 0$ for all $1 \leq i \leq n$ and $1 \leq j \leq N$. In addition, we establish in Corollary \ref{cor:combined estimate} below that there is a rank-dependent upper bound for $C_n$: 
\[
\gp(A) \leq  \sqrt{n}\cdot \min\set{\rk(A), \sqrt{n}} \cdot \norm{A}_{1,1}
\text{ for all } A \in \psd^n. 
\]

In Section \ref{sect:infinite-dim problem}, we explore the infinite-dimensional setting and study Problem \ref{prob:Feichtinger}. We establish several results of various flavors. For simplicity, we use $\Gamma_+$ to denote the family of operators that admit a strong $\ell_2(\ell_1)$ factorization, i.e., satisfying \eqref{eq:condition-strongl2-intro}.

To begin with, we construct an example in Theorem \ref{thm:Feichtinger->matrix} that answers Problem \ref{prob:Feichtinger} negatively. This relies on a joint effort with \cite{BMS24}, which establishes \eqref{eq: Cn BMS}.

In section \ref{sect:topology and density}, we introduce a topology $\tau$ on self-adjoint positive operators induced by (the infinite-dimensional variant of) $\gp$, which characterizes ``approximation from above'' with respect to the partial ordering on the positive cone. 
We study the properties of $\tau$ in Proposition \ref{prop:topology, basis, non-second countability}, and in particular, show that $\tau$ is not second countable, which is a consequence of the example constructed in Theorem \ref{thm:Feichtinger->matrix}. In Theorem \ref{thm:density}, we show that the family of those satisfying \eqref{eq:condition-nuclear-intro} is the $\tau$-closure of $\Gamma_+$.

The main result of Section \ref{sect:holomorphic action}, Theorem \ref{thm:holomorphic action}, states that $\Gamma_+$ is invariant under the action of positive-coefficient subalgebra of the analytic Wiener algebra. The proof of Theorem \ref{thm:holomorphic action} makes use of two ingredients of independent interest: the first ingredient is a variant of the celebrated Grothendieck-Pietsch factorization through $\ell_2$ using the square root operator (which can be explicitly given via the spectral theorem); the second ingredient states that the quadratic action of $\ell_1$ endomorphisms also leaves $\Gamma_+$ invariant.

In Section \ref{sect:2-summing}, we provide new insight into the relationships between $\Gamma_+$-membership, summability of eigendecomposition, and $2$-summing operators $\ell_2 \to \ell_1$. In Theorem \ref{thm:2-summing factorization, sufficient Type A}, we show that if $A = TT^*$ where $T:\ell_2\to \ell_1$ is $2$-summing, then any eigendecomposition of $A$ is a strong $\ell_2(\ell_1)$ sequence (so $A\in \Gamma_+$). Finally, we show in Theorem \ref{thm:square root 2-summing, sufficient, square root diagonal} that $A^{1/2}:\ell_2 \to \ell_1$ is $2$-summing if and only if $\sum_k \sqrt{A_{kk}} < \infty$, thus providing a complete characterization of a special subfamily of $\Gamma_+$.

\subsection*{Acknowledgment}

The first author has been supported in part by
the National Science Foundation under grant DMS-2108900 and by the Simons Foundation. He also thanks Afonso Bandeira, Dustin Mixon, and Stefan Steinerberger for providing an early copy of \cite{BMS24}.

The second author is partially supported by a grant from
the City University of Hong Kong (Project No. 7200844).
Part of this material is also based on work supported by the National Science Foundation under Grant No. DMS-1439786 while the second author was in residence at the Institute for Computational and Experimental Research in Mathematics in Providence, RI, during the fall 2022 semester. He thanks Carsten Schuett for his lecture and discussion on $p$-summing operators. 

Both authors thank the anonymous referee for their valuable comments and suggestions.

\subsection*{Notations}

\paragraph{Banach spaces and operators} Let $X$ be a Banach space. We use $B_X$ to denote the closed unit ball of $X$. We use $\jbrac{\cdot,\cdot}_{X^*,X}$ to denote the duality pairing. When the choice of $X$ is clear from context, we simply write $\jbrac{\cdot,\cdot}$. For a pair of Banach spaces $X$ and $Y$, we use $\L(X,Y)$ to denote the Banach space of bounded linear operators from $X$ to $Y$ with the operator norm $\norm{\cdot}_{X\to Y}$.
Let $x^* \in X^*, y \in Y$, we write $yx^* $ or $ x^*\otimes y : X \to Y$ to denote the rank-one operator defined by $[{x^*\otimes y}](v) = \jbrac{x,v}y$.   

\paragraph{Finite dimensional normed spaces} For $x \in \R^n$ or $\mathbb{C}^n$, we define $\norm{x}_p := (\sum_{k=1}^n\abs{x(k)}^p)^{1/p}$ with the usual modification for $p = \infty$. Here, $x(k)$ is the $k$-th entry of $x$. We use $\ell_p^n$ to denote $\R^n$ or $\C^n$ equipped with $\norm{\cdot}_p$. We use $\Sigma_p^n$ to denote the unit sphere in $\ell_p^n$. 

\paragraph{Sequence spaces} We write $\ell_p$ for $\ell_p(\mathbb{N})$ with $\norm{\cdot}_p$ defined similarly. 
\  We use $\set{\delta_k}$ to denote the standard basis of $\ell_2$. We use $c_0\subset \ell_\infty$ to denote the Banach subspace of sequences converging to zero. Recall that $c_0^* \simeq \ell_1$.  

\paragraph{Matrices and operators} We use $\M^n$ to denote the vector space of matrices with real or complex entries, $\sym^n$ to denote the subspace of symmetric or Hermitian matrices, and $\psd^n\subset \sym^n$ to denote the cone of positive-semidefinite matrices. We also use the notation $A \succ 0$ when $A: \ell_2 \to \ell_2$ is positive-semidefinite as a bilinear form on $\ell_2$, i.e., $\ip{Ax,x} \geq 0$ for all $x \in \ell_2$.

\section{The norms $\g$ and $\gp$}
\label{sect:nuclear}

Let $p,q \in [1,\infty]$ and $A: \ell_p \to \ell_q$. The operator norm of $A$ is defined to be
\begin{equation*}
    \norm{A}_{p\to q} := \sup_{\norm{z}_p \leq 1}\norm{Az}_q.
\end{equation*}
We have
\begin{equation*}
    \norm{A}_{r \to q} \leq \norm{A}_{\infty \to 1}
    \text{ for all } 1 \leq r,q \leq \infty.
\end{equation*}
We also define the ``entry-wise'' norm
\begin{equation*}
    \norm{A}_{q,p} := \brac{
    \sum_{k}
    \brac{\sum_{l} \abs{\jbrac{A\delta_k,\delta_l}}^{q} }^{\frac{p}{q}}
    }^{1/p} \in [0,\infty],
\end{equation*}
with suitable modification when $p$ or $q$ is $\infty$. 
If $A$ is a finite-dimensional matrix, this is simply $\norm{A}_{q,p} = \brac{\sum_{k = 1}^{n} \norm{\mathrm{Col}_kA}_q^p}^{1/p}$
with suitable modification when $p$ or $q$ is $\infty$.

\subsection{Nuclear operators}

Let $X$ and $Y$ be Banach spaces. A bounded linear operator $A:X \to Y$ is nuclear if there are sequences 
    $(x_k^*) \subset X^* $ and $(y_k) \subset Y $ such that
    \begin{equation}
        A = \sum_{k = 1}^\infty x_k^*\otimes y_k
        \label{eq:nuclear-representation}
    \end{equation}
    in the strong operator topology. The nuclear norm (or projective norm) of the operator $A$ is defined as
    \begin{equation}
        \g(A:X \to Y) = \inf \sum_{k=1}^{\infty} \norm{x_k}_{X^*}\norm{y_k}_{Y}
        \label{eq:def:gamma}
    \end{equation}
    where the infimum is taken over all representations of the form \eqref{eq:nuclear-representation}. 

   Suppose $A$ is nuclear and let $A = \sum_k x_k^*\otimes y_k$ be a strongly convergent series with $\sum_{k}\norm{x_k}_{X^*}\norm{y}_Y < \infty$. Then the series is also convergent in the operator norm. Thus, every nuclear operator is a limit of finite rank operators, and hence compact. 

Nuclear operators were originally introduced by Grothendieck in his doctoral dissertation \cite{Groth54}. See \cite{DJT95,nuclear-book} for a systematic treatment. The nuclear norm is usually denoted by $\nu$ in the standard literature, but we keep the notation $\g$ to be consistent with the previous work \cite{BORWZ} on Problem \ref{prob:matrix}.

\begin{proposition}\label{prop:nuclear-vs-entry}
    Let $1 \leq p < \infty$ and $A : c_0 \to \ell_p$. Then 
    \begin{equation*}
        \g(A:c_0 \to \ell_p)= \norm{A}_{p,1}
    \end{equation*}
\end{proposition}

\begin{proof}
    Without loss of generality, we assume all the quantities of interest are finite. 
    
    For the first inequality, we write $A$ as the strongly convergent series
    \begin{equation}
        A = \sum_{k}\delta_k^*\otimes A\delta_k
    \end{equation}
    with $\delta_k^* \in c_0^* \simeq \ell_1$ and $A\delta_k \in \ell_p$. Consequently,
\begin{equation*}
    \g(A) \leq \sum_{k} \norm{A\delta_k}_p\norm{\delta_k}_{1}  = \norm{A}_{p,1}.
\end{equation*}

For the second inequality, let $\epsilon>0$ and let $(x_k^*) \subset c_0^*\simeq \ell_1$, $(y_k)\subset \ell_p$ with 
\begin{equation*}
    A = \sum_{k} x_k^*\otimes y_k
    \text{ and }
    \sum_k \norm{x_k}_{1}\norm{y_k}_p \leq \g(A) + \epsilon.
\end{equation*}
Note that
\begin{equation*}
    \norm{x^*\otimes y}_{p,1} = \sum_{i }\norm{x(i)y}_p \leq \norm{x}_1\norm{y}_p.
\end{equation*}
Therefore,
\begin{equation*}
    \norm{A}_{p,1} \leq \sum_{k}\norm{x_k^*\otimes y_k}_{p,1} \leq \sum_{k}\norm{x_k}_1\norm{y_k}_p \leq \g(A) + \epsilon.
\end{equation*}
\end{proof}

Note that the proof for the (finite-dimensional) special case (with $p = 1$) was given in \cite[Lemma 2.7]{BORWZ}. In particular, we record for repeated use throughout this paper that, for any matrix $A$, viewed as an operator $\ell_\infty^n \to \ell_1^n$,
\begin{equation}
    \begin{split}
        \g(A) = \g(A:\ell_\infty^n \to \ell_1^n) 
        &= \inf\set{\sum_{k = 1}^N \norm{x_k}_1\norm{y_k}_1 : A = \sum_{k = 1}^N x_k^*\otimes y_k} \\
        &= \sum_{k,l = 1}^n|A_{kl}| = \norm{A}_{1,1}.
    \end{split}
    \label{eq:gamma=oneone}
\end{equation}

\subsection{Basic estimates and continuity of $\gp$}

From here till section \ref{sect:infinite-dim problem}, we concern ourselves with finite-dimensional vector spaces and deal primarily with matrices.

Recall that, for $A \in \psd^n$,
we define
\begin{equation}
    \gp(A) := \inf \set{\sum_{k = 1}^N \norm{x_k}_{1}^2 : A = \sum_{k = 1}^N x_kx_k^*}.
    \label{eq:def:gamma+}
\end{equation}

\begin{proposition}[\cite{BORWZ}]
\label{prop:gamma+ basics}
    The functional $\gp : \psd^n \to \R$ has the following properties.
    \begin{enumerate}[(A)]
        \item  $\gp$ is sub-additive and positive-homogeneous on $\psd^n$.
        \item For $A \in \psd^n$, $\gp(A) = \min\set{\sum_{k = 1}^{n^2 }\norm{x_k}_1^2 : A = \sum_{k = 1}^{n^2} x_kx_k^*}$. In particular, $\gp$ can be achieved by a certain finite decomposition. The maximum number of terms is $n^2$ in the complex case, and $n(n+1)/2$ in the real case (that is, when $A\in\R^{n\times n})$. 
    \end{enumerate}
\end{proposition}

We begin with an elementary but useful estimate. 

\begin{proposition}\label{prop:ntrace}
    Let $\g(A) = \g(A:\ell_\infty^n \to \ell_1^n)$.
    Given $A \in \psd^n$, we have 
\begin{equation}
    \label{eq:2.3}
\g(A) \leq \gp(A) \leq n\tr(A) \leq n\norm{A}_{1,1} = n\g(A)
\end{equation}
\end{proposition}

\begin{proof}
    The only nontrivial inequality is $\gp(A)\leq n\tr(A)$.
    Fix $A \in \psd^n$ and let $(x_k)_{k = 1}^N$ be any factorization $A =\sum_k x_kx_k^*$. Let $1_n = (1,\cdots, 1)^t$. Since
    \begin{equation*}
        \norm{x_k}_1 = \norm{\ip{x_k, 1_n}}_1 \leq \norm{1_n}_2\norm{x_k}_2 = \sqrt{n}\norm{x_k}_2,
    \end{equation*}
    we have
    \begin{equation*}
        \gp(A) \leq \sum_k \norm{x_k}_1^2\leq n\sum_k \norm{x_k}_2^2 = n\sum_k x_k^*(x_k) = n\tr(A).
    \end{equation*}
\end{proof}

An alternative estimate is also available as follows.

\begin{proposition}\label{prop:sqrt(n)rank}
    Given $A \in \psd^n\setminus\set{0}$, we have
    \begin{equation}
        \gp(A)  \leq  \sqrt{n} \cdot \rk(A)\cdot \norm{A}_{2\to 1} \leq 
        \sqrt{n}\cdot \rank(A) \cdot \norm{A}_{1,1}.        
    \end{equation}
\end{proposition}

\begin{proof}
    Let $A = \sum_k^{r}  x_kx_k^*$ be the spectral factorization, $r = \rank( A)$, and $Ax_k = \lambda_kx_k$ with $\lambda_k > 0$. Let $e_k = \lambda_k^{-1/2}x_k$. 
    Then
    \begin{equation*}
        \lambda_k\norm{x_k}_1 = \norm{Ax_k}_1 \leq \norm{A}_{2\to 1}\norm{x_k}_2.
    \end{equation*}
    Rearrange, we see that
    \begin{equation*}
        \norm{x_k}_1 \leq \norm{A}_{2\to 1}\lambda_k^{-1/2}\norm{e_k}_2 =  \norm{A}_{2\to 1}\lambda_k^{-1/2}.
    \end{equation*}
    By H\"older's inequality, we have $\norm{x_k}_1 \leq \sqrt{n}\norm{x_k}_2$. Therefore,
    \begin{equation*}
        \begin{split}
            \gp(A) &\leq \sum_{k = 1}^r \norm{x_k}_1^2\leq \sqrt{n}\norm{A}_{2\to 1}\sum_{k = 1}^{r}\norm{e_k}_2^2 = \sqrt{n}\norm{A}_{2\to 1}\tr\brac{\sum_{k=1}^re_ke_k^*} \\
            &= \mathrm{rk}(A)\cdot \sqrt{n}\cdot \norm{A}_{2\to 1}\leq \mathrm{rk}(A)\cdot \sqrt{n}\cdot \norm{A}_{1,1}.
        \end{split}
    \end{equation*}
\end{proof}

\begin{corollary}\label{cor:combined estimate}
    Given $A \in \psd^n\setminus \set{0}$, we have
    \[
    \gp(A) \leq  \sqrt{n}\cdot \min\set{\rk(A), \sqrt{n}} \cdot \norm{A}_{1,1}.
    \]
\end{corollary}
\begin{proof}
    The estimate follows directly from Propositions \ref{prop:ntrace} and \ref{prop:sqrt(n)rank}.
\end{proof}

\begin{theorem}\label{thm:continuity}
    $\gp$ is continuous on $(\psd^n,\norm{\cdot}_{1,1})$.
\end{theorem}

\begin{remark}
    The continuity of $\gp$ on the relative interior of $\psd^n$ (the set of strictly positive definite matrices) was shown in \cite{BORWZ}. However, this does not directly imply the continuity on $\psd^n$. The paper \cite{BORWZ} establishes additionally that $\gp$ is locally Lipschitz on this relative interior. We do not know if this remains true on the entire $\psd^n$. 
\end{remark}

Since $\sym^n$ is a finite-dimensional vector space, it suffices to prove continuity with respect to $\norm{\cdot}_{2\to 2}$. The proof of Theorem \ref{thm:continuity} depends on the following lemma.

\begin{lemma}\label{lem:continuity}
    Let $A \in \psd^n$ be of rank $r > 0$. Let $\lambda_r > 0$ denote the $r$-th eigenvalue of $A$.
    \begin{enumerate}[(A)]
        \item Let $P_{A,r}$ denote the orthogonal projection onto the range or $A$. For any $\epsilon \in (0,1)$ and $B \in \psd^n$ with $\norm{A-B}_{2\to 2} \leq \frac{\epsilon}{1-\epsilon}\lambda_r$, we have
        \begin{equation*}
        \eqindent
            A - (1-\epsilon)P_{A,r}BP_{A,r} \in \psd^n.
        \end{equation*}
        \item For any $\epsilon \in (0,\frac{1}{2})$ and $B \in \psd^n$ with $\norm{A-B}_{2\to 2} \leq \epsilon\lambda_r$, we have
        \begin{equation*}
            \eqindent
            B - (1-\epsilon)P_{B,r}AP_{B,r}\in \psd^n,
        \end{equation*}
        where $P_{B,r}$ is the orthogonal project onto the top $r$ eigenspace of $B$.
    \end{enumerate}
\end{lemma}

\begin{proof}
    For (A), write $P = P_{A,r}$, let $\Delta = A - (1-\epsilon)PBP$, and let $z \in \ell_2^n$. We have
    \begin{equation*}
        \begin{split}
            \jbrac{\Delta x, x} 
            &= \jbrac{APx, Px} - (1-\epsilon)\jbrac{BPx, Px}\\
            &= \jbrac{ \left[A - (1-\epsilon)B\right]Px, Px }\\
            &= (1-\epsilon)\jbrac{(A-B)Px, Px} + \epsilon\jbrac{APx, Px} \\
            &\geq \epsilon \lambda_r\norm{Px}_2^2 - (1-\epsilon)\norm{A-B}_{2\to 2}\norm{Px}_2^2\\
            \brac{\norm{A-B}_{2\to 2} \leq \frac{\epsilon}{1-\epsilon}\lambda_r} 
            \quad &\geq 0. 
        \end{split}
    \end{equation*}

    For (B), write $P = P_{B,r}$, let $\Delta = B - (1-\epsilon)PAP$, and let $C = B - PBP \in \psd^n$. Let $\mu$ be the $r$-th eigenvalue of $B$. Note that
    \begin{equation*}
        \abs{\mu - \lambda_r} \leq \norm{A-B}_{2\to 2} \leq \epsilon\lambda_r \Longrightarrow
        \mu \geq (1-\epsilon)\lambda_r.
    \end{equation*}
    For any $x \in \ell_2^n$,
    \begin{equation*}
        \begin{split}
            \jbrac{\Delta x, x}
            &= \jbrac{Cx, x} + 
            \jbrac{BPx,Px}
            - (1-\epsilon)\jbrac{APx, Px}
            \\
            &= \jbrac{Cx,x} + \epsilon\jbrac{BPx,Px} + (1-\epsilon)\jbrac{(B-A)Px,Px}
            \\
            &\geq \jbrac{Cx,x} + \brac{\epsilon\mu - (1-\epsilon)\norm{A-B}_{2\to 2}}\norm{Px}_2^2
            \\
            \brac{\norm{A-B}_{2\to 2}\leq \epsilon\lambda_r \leq \frac{\epsilon}{1-\epsilon}\mu}
            &\geq 0.
        \end{split}
    \end{equation*}
\end{proof}

\begin{proof}[Proof of Theorem \ref{thm:continuity}]
    Let $B_j \to A$ in $\psd^n$ (with respect to $\norm{\cdot}_{2 \to 2}$). We want to show $\gp(B_j) \to \gp(A)$.

    Suppose $A = 0$, then $B_j \to 0$. By Proposition \ref{prop:ntrace} (or \ref{prop:sqrt(n)rank}),
    \begin{equation*}
        0 \leq \gp(B_j) \leq n\tr(B_j) \to 0 \text{ as } n\to \infty.
    \end{equation*}
    
    Suppose $A \neq 0$. Let $A = \sum_k^{N} x_kx_k^*$ be an optimal decomposition for $\gp$. Let $\epsilon \in (0,1/2)$ be arbitrary. Let $J = J(\epsilon)$ be so large such that
    \begin{equation*}
        \norm{A-B_j}_{1,1} \leq \epsilon\lambda_r \quad \forall j  > J.
    \end{equation*}
    Let $B_j = \sum_{k = 1}^N w^{(j)}_k(w^{(j)}_k)^*$ be an optimal decomposition for $\gp$. Set
    \begin{equation*}
        \Delta_j := A - (1-\epsilon)P_{A,r}B_j P_{A,r}.
    \end{equation*}
    By Lemma \ref{lem:continuity}, for any $j > J$,
    \begin{equation*}
        \gp(A) \leq (1-\epsilon)\gp(P_{A,r}B_jP_{A,r}) + \gp(\Delta_j) \leq (1-\epsilon)\sum_{k = 1}^N \norm{P_{A,r}w^{(j)}_k}_1^2 + n\cdot \tr(\Delta_j).
    \end{equation*}
    Replacing by a subsequence, so that $w^{(j_\ell)}_k \to y_k \in \ell_1^n$ for every $k$ and
    \begin{equation*}
        \lim_{\ell\to\infty}\gp(B_{j_\ell})  = \liminf_{j \to \infty}\gp(B_j).
    \end{equation*}
    We then have
    \begin{equation*}
        \lim_{\ell\to\infty} P_{A,r}w^{(j_\ell)}_k = P_{A,r}y_k = y_k
    \end{equation*}
    and
    \begin{equation*}
        \lim_{\ell \to \infty}\sum_{k = 1}^N \norm{P w^{j_\ell}_k}_1^2 
        =
        \lim_{\ell\to\infty} \sum_{k = 1}^N \norm{w^{j_\ell}_k}_1^2 = \liminf_{j\to\infty}\gp(B_j).
    \end{equation*}

    On the other hand,
    \begin{equation*}
        \lim_{j\to\infty}\tr(\Delta_j) = \epsilon\tr(A).
    \end{equation*}
    Therefore,
    \begin{equation*}
        \gp(A) \leq (1-\epsilon)\liminf_{j\to\infty}\gp(B_j) + \epsilon\cdot \tr(A).
    \end{equation*}
    Since $\epsilon > 0$ is arbitrary, we can conclude that
    \begin{equation*}
        \gp(A) \leq \liminf_{j\to\infty}\gp(B_j).
    \end{equation*}

    The proof of the reverse inequality is similar by using Lemma \ref{lem:continuity}(B), setting $\Delta_j = B_j - (1-\epsilon)P_{B_j,r}AP_{B_j,r}$, estimating 
    \begin{equation*}
        \gp(B_j) \leq (1-\epsilon)\sum_{k = 1}^N\norm{P_{B_j,r}x_k}_1^2 + n\tr(\Delta_j),
    \end{equation*}
    and taking the $\limsup$:
    \begin{equation*}
        \limsup_{j \to \infty}\gp(B_j) \leq (1-\epsilon)\gp(A) + n^2\epsilon\norm{A}_{2\to 2}.
    \end{equation*}
    
\end{proof}

\subsection{Some special cases}

We discuss some special cases where $\gp = \g$. We summarize the results of \cite[Section 3]{BORWZ} in the following proposition.

\begin{proposition}\label{prop:special-case-1}
    $\gp(A) = \norm{A}_{1,1}$ in the following three cases.
    \begin{enumerate}[\text{Case} 1.]
        \item $A \in \psd^n$ has rank one.
        \item $A \in \psd^2$, i.e., $A$ is a $2\times 2$ positive-semidefinite matrix;
        \item $A \in \psd^n$ is diagonally dominant, i.e., $A(k,k) \geq \sum_{l\neq k}\abs{A(k,l)}$ for all $k \in [n]$.
    \end{enumerate}
\end{proposition}

Recall that a matrix $A \in \psd^n(\R)$ is {\em completely positive} if there exists $X \in \R_+^{n\times N}$ such that $A = XX^*$, i.e., $A$ admits a symmetric nonnegative matrix factorization (NMF). 

\begin{proposition}\label{prop:special-case-CP}
    Suppose $A \in \psd^n(\R)$ is completely positive. Then $\gp(A) = \norm{A}_{1,1}$.
\end{proposition}

\begin{proof}
    We write $\g(A)$ for $\g(A:\ell_\infty^n\to\ell_1^n)$.
    Thanks to Proposition \ref{prop:nuclear-vs-entry}, we have $\norm{\cdot}_{1,1} = \gamma(\cdot)$.
    By the definitions of $\g$ and $\gp$ in \eqref{eq:gamma=oneone} and \eqref{eq:def:gamma+}, we always have $\norm{\cdot}_{1,1} = \g \leq \gp$. It suffices to show the reverse inequality.
    
    Fix a symmetric nonnegative factorization $A = XX^*$ with $X \in \R_+^{n\times N}$. Setting $x_k := X\delta_k$ where $\set{\delta_k}_{k = 1}^n$ is the standard basis for $\Rn$, we have $A = \sum_{k = 1}^N x_kx_k^*$. 
    Then
    \begin{equation*}
        \begin{split}
            \gp(A) &\leq \sum_{k=1}^N \norm{x_k}_1^2 = \sum_{k = 1}^N \brac{\sum_{i = 1}^n x_k(i)}^2 
            =  \sum_{i,j = 1}^n\sum_{k = 1}^N x_k(i)x_k(j) = \sum_{i,j = 1}^n A(i,j) = \norm{A}_{1,1} = \g(A).
        \end{split}
    \end{equation*}
\end{proof}

It is known that for $n \leq 4$, if $A \in \psd^n(\R)\cap \R_+^{n\times n}$ then $A$ is completely positive. It was shown by Hall and Newman \cite{HallNewman1963-completelypositive} that this phenomenon fails for $n \geq 5$ . We remark that Proposition \ref{prop:special-case-CP} is related to a special case of
Theorem \ref{thm:2-summing factorization, sufficient Type A} in a later section.

\section{Duality}
\label{sect:duality}
\renewcommand{\M}{\mathcal{B}}

Let $\Sone$ denote the unit sphere of $\ell_1^n$. Let $\M_+(\Sone)$ denote the cone of positive Borel measures on $\Sone$.

\begin{theorem}\label{thm:gamma+superresolution}
    For any $A \in \psd^n(\mathbb C)$,
    \begin{equation}
        \gp(A) = \inf\set{ \int_{\Sone}d\mu : A = \int_{\Sone}xx^*d\mu,\, \mu \in \M_+(\Sone) }
        .
        \label{eq:def:gamma+ measure}
    \end{equation}
    The infimum can be replaced by a minimum. 
\end{theorem}

\begin{proof}
    Let $p^*:=  \inf\set{ \int_{\Sone}d\mu : A = \int_{\Sone}xx^*d\mu,\, \mu \in \M_+(\Sone) }$. Assume $A = \sum_{k = 1}^N x_kx_k^*$ is an optimal decomposition for $\gp$ (in the sense of \eqref{eq:def:gamma+}), the existence of which is guaranteed by Proposition \ref{prop:gamma+ basics}(B). Then 
    \begin{equation*}
        \mu(x) := \sum_{k = 1}^N\norm{x_k}_1^2 \delta(x-\tfrac{x_k}{\norm{x_k}_1}) \in \M_+(\Sone)
    \end{equation*}
    satisfies $A = \int_{\Sone}xx^*d\mu$, so $p^*\leq \gp(A)$.

    For the reverse inequality, we first note that the infimum is achieved, thanks to the Banach-Alaoglu Theorem. Let $\mu^*$ be an optimal measure achieving $p^*$. Fix $\epsilon > 0$. Let $(U_k)_{1 \leq k \leq \kmax}$ be a sufficiently fine disjoint partition of $\Sone$ such that there exist $z_1, \cdots, z_{\kmax} \in \Sone$
    \begin{equation*}
        U_k \subset B_\epsilon(z_k)\cap \Sone
        \text{ for all } k. 
    \end{equation*}
    For each $k = 1, \cdots, \kmax$, define
    \begin{align*}
         x_k &:= \frac{1}{\mu(U_k)}\int_{U_k}xd\mu^*(x) \in B_\epsilon(z_k),\\
         g_k &:= \sqrt{\mu^*(U_k)}x_k, \text{ and }\\
         R_k &:= \int_{U_k}(x-x_k)(x-x_k)^*d\mu^*(x) = \int_{U_k}xx^*d\mu^*(x) - \mu^*(U_k)x_kx_k^* \geq 0.
    \end{align*}
    Summing over $k$, with $R = \sum_{k}^{\kmax}R_k$, we have
    \begin{equation}
        A = \int_{\Sone}xx^*d\mu^*(x) \leq \sum_{k = 1}^{\kmax} g_kg_k^* + R.
        \label{eq:thm:gamma+superresolution-1}
    \end{equation}
    Applying Proposition \ref{prop:gamma+ basics}(A) and \ref{prop:ntrace} to \eqref{eq:thm:gamma+superresolution-1}, we see that
    \begin{equation}
        \gp(A) \leq \sum_{k = 1}^{\kmax}\norm{g_k}_1^2 + \gp(R) \leq 
        \sum_{k = 1}^{\kmax} \mu^*(U_k)\norm{x_k}_1^2 + n\tr(R).
        \label{eq:thm:gamma+superresolution-2}
    \end{equation}
    Since $\norm{x_k - z_k}_1 \leq \epsilon$ and $\norm{x - x_k}_1 \leq 2\epsilon$ for every $k$ and $x \in U_k$, we have
    \begin{equation}
        \norm{x_k}_1 \leq 1 + \epsilon
        \text{ and }
        \tr(R_k) \leq 4\mu^*(U_k)\epsilon^2.
        \label{eq:thm:gamma+superresolution-3}
    \end{equation}
    It follows from \eqref{eq:thm:gamma+superresolution-2} and \eqref{eq:thm:gamma+superresolution-3}
    that 
    \begin{equation*}
        \gp(A) \leq \mu^*(\Sone) + (2\epsilon + \epsilon^2 + 4n\epsilon^2)\mu^*(\Sone).
    \end{equation*}
    Since $\epsilon > 0$ is arbitrary and $n$ is fixed, we can conclude that $\gp(A) \leq p^*$.

    Thanks to Proposition \ref{prop:gamma+ basics}(B), the infimum can be achieved. 
\end{proof}

\subsection{Strong duality for $\gp$}

Recall that $\M_+(\Sone)$ denotes the cone of positive Borel measures on $\Sone$, the unit sphere in $\ell_1^n$. Let $\M(\Sone)$ denote the space of signed Borel measures on $\Sone$. By the Riesz-Markov-Kakutani representation theorem,
\begin{equation*}
    C(\Sone)^* \cong \M(\Sone).
\end{equation*}
We also have the duality of the positive cones
\begin{equation*}
    C_+(\Sone)^*\cong \M_+(\Sone)
\end{equation*}
where $C_+(\Sone)$ is the cone of nonnegative continuous functions. 

Consider the dual pairs $(\M(\Sone), C(\Sone))$ and $(\sym^n, \sym^n)$ with the duality pairing, i.e., we equip $\M$ with the weak star topology, $C(\Sone)$ with the weak topology, and $\sym^n$ with the Euclidean topology.

Consider the map
\begin{equation*}
    \Phi: \M(\Sone) \to \sym^n 
    \quad,\quad \mu \mapsto \Phi(\mu) = \int_{\Sone}xx^*d\mu.
\end{equation*}
The adjoint of $\Phi$ is given by
\begin{equation*}
    \Phi^*: \sym^n \to C(\Sone),\,
    \quad T\mapsto \Phi^*(T)(x) = \tr[(x^*\otimes x)\circ T] = \jbrac{Tx,x}.
\end{equation*}
By linear duality theory, the following functional on $\psd^n$
\begin{equation*}
    \delta_+(A) := \sup\set{\tr(AT) : T \in \sym^n \text{ and } \jbrac{Tx,x} \leq 1 \,,\forall x \in \Sone }
\end{equation*}
is the dual of the (measure-theoretic) linear program associated with $\gp$.

\begin{theorem}\label{thm:strong duality}
    For all $A \in \psd^n$, $\delta_+(A) = \gp(A)$. 
\end{theorem}

We will break down the proof of Theorem \ref{thm:strong duality} into several steps in the rest of the section.

\begin{lemma}\label{lem: delta+ leq gamma+}
 $\delta_+(A)\leq \gp(A)$ for all $A \in \psd^n$, $n \in \mathbb{N}$.
\end{lemma}

\begin{proof}
    Fix $A \in \psd^n$ and let $\epsilon > 0$. Let $A = \sum_{k = 1}^N x_k x_k^*$ be a decomposition such that $\sum_k \norm{x_k}_1^2 = \gp(A)$. Note that for any $T = T^*$ with $\jbrac{Tx,x}\leq 1$ for all $x \in \Sone$,
    \begin{equation*}
        \tr(AT) = \tr\brac{\sum_{k}x_kx_k^* \cdot T} = \sum_k \jbrac{T x_k,x_k} = \sum_k \norm{x_k}_1^2 \jbrac{\frac{Tx_k}{\norm{x_k}_1},\frac{x_k}{\norm{x_k}_1}} \leq \sum_{k}\norm{x_k}_1^2 \leq \gp(A).
    \end{equation*}
\end{proof}

\begin{theorem}[Hyperplane Separation Theorem]\label{thm:Hahn-Banach sep}
    Let $K\subset \R^d$ be a closed convex set and let $x_0 \in \R^d \setminus K$. Then there exists a bounded linear functional $\phi$ and a real number $\alpha$ such that
    \begin{equation*}
        \phi(x_0) < \alpha < \phi(x)
        \quad \text{ for all } x \in K.
    \end{equation*}
\end{theorem}

See \cite[Example 2.20]{BV-convexbook}. 

For the reverse inequality $\gp \leq \delta_+$, consider the convex body
\begin{equation}
    \mathcal{K}=\set{ (\Phi(\mu),\jbrac{\mu,1} + r) : \mu \in \M_+(\Sone),\, r \geq 0 }\subset \sym^n \times \R
    \label{eq:def:K}
\end{equation}
equipped with the induced (Euclidean) topology.

The next lemma is inspired by \cite{Anderson1983}. We invite the interested readers to \cite{Anderson1983} and references therein for background on linear programming duality in infinite-dimensional spaces.

\begin{lemma}\label{lem: delta+ geq gamma+}
    Let $A \in \psd^n$ and let $\mathcal{K}$ be as in \eqref{eq:def:K}.
    Then $(A,\delta_+(A)) \in \overline{\mathcal{K}}$.
\end{lemma}

\begin{proof}
   Let $M = \delta_+(A)$.
    Suppose towards a contradiction, that $(A,M)\notin \overline{\mathcal{K}}$. Then $(A,M)$ is separated from $\mathcal{K}$ by a hyperplane in the sense of Theorem \ref{thm:Hahn-Banach sep}, i.e., there exists $(T_0, \lambda) \in \sym^n\times \R$ such that
    \begin{equation}
        \tr(AT_0) + \lambda M < \tr\brac{\int_{\Sone}xx^*d\mu \cdot T_0} + \lambda\int_{\Sone}d\mu + \lambda r
        \text{ for all } \mu \in \M_+(\Sone) \text{ and } r\geq 0.
        \label{eq:lem:duality-2}
    \end{equation}
    By setting $\mu = 0$ and $r = 0$ in \eqref{eq:lem:duality-2}, we see that
    \begin{equation}
        \tr(AT_0) + \lambda M < 0.
        \label{eq:lem:duality-contradiction}
    \end{equation}
    Note that the right hand side of \eqref{eq:lem:duality-2} must also be nonnegative, i.e.,
    \begin{equation}
        \tr\brac{\int_{\Sone}xx^*d\mu \cdot T_0} + \lambda\int_{\Sone}d\mu + \lambda r \geq 0\quad
        \text{ for all } \mu \in \M_+(\Sone) \text{ and } r\geq 0.
        \label{eq:lem:duality-3}
    \end{equation}

    By setting $r = 0$ in \eqref{eq:lem:duality-3} and let $\mu$ range over all possible Borel measures, we have 
    \begin{equation}
        \Phi^*(T)(x) = \jbrac{T_0x,x} \geq -\lambda
        \quad\text{ for all } x \in \Sone.
        \label{eq:lem:duality-4}
    \end{equation}
    
    By setting $\mu = 0$ in \eqref{eq:lem:duality-3}, we have 
    \begin{equation*}
        \lambda \geq 0.
    \end{equation*}

     \underline{Suppose $\lambda > 0$.} Then $-\lambda^{-1}T_0 \in \sym^n$ and $\jbrac{-\lambda^{-1}T_0x,x} = -\lambda^{-1}\jbrac{T_0x,x} \leq 1$ by \eqref{eq:lem:duality-4}. Therefore, $-\lambda^{-1}T_0$ is feasible for $\delta_+$. By \eqref{eq:lem:duality-contradiction},
            \begin{equation*}
                \tr(A(-\lambda^{-1}T_0)) > M,
            \end{equation*}
            contradicting $M$ being the supremum.

    \underline{Suppose $\lambda = 0$.} Let $\tilde{T} \in \sym^n$ with $\jbrac{Tx,x} \leq 1$ for all $z \in \Sone$. It is clear that $\tilde{T}-\alpha T_0 \in \sym^n$ for all $\alpha \geq 0$. Since $\Phi(T_0)(z) \geq 0$ for all $z \in \Sone$, we have $\jbrac{z,(\tilde{T} -\alpha T)z} \leq 1$ for all $z \in \Sone$ and $\alpha\geq 0$. Therefore, $\tilde{T}-\alpha T_0$ is feasible for $\delta_+$ for all $\alpha \geq 0$. Since $\tr(A(\tilde{T}-\alpha T_0)) \leq  \delta_+(A) = M < \infty$, we must have $\tr(AT_0) \geq 0$. This contradicts \eqref{eq:lem:duality-contradiction}.

In either case, we arrive at a contradiction. Therefore, $(A,M) \in \overline{\mathcal{K}}$.

\end{proof}

\begin{lemma}\label{lem: K closed}
    Let $\mathcal{K}$ be as in \eqref{eq:def:K}. $\mathcal{K}$ is closed. 
\end{lemma}

\begin{proof}
    Let $(A_j,M_j)_j$ be a sequence in $\mathcal{K}$ converging to $(A,M) \in \sym^n \times \R$. By definition, there exists a sequence $\mu_j \in \M_+(\Sone)$ and $r_j = M_j - \int_{\Sone}d\mu_j \in \R$, such that
    \begin{equation*}
        A_j = \int_{\Sone}xx^*d\mu_j.
    \end{equation*}
    In particular, the sequence $\mu_j$ is bounded. By Banach-Alaoglu, we can replace $\mu_j$ by a subsequence, such that $\mu_j \to \mu\in\M_+(\Sone)$ in the weak-star topology. 
    Integrating against $xx^*$, we have
    \begin{equation*}
        \int_{\Sone}xx^*d\mu = A.
    \end{equation*}
    Integrating against the constant function $1$, we have
    \begin{equation*}
        M = \int_{\Sone}d\mu + \lim_{j\to\infty} r_j.
    \end{equation*}
    The last limit is nonnegative since $r_j \geq 0$ for all $j$. Therefore, $(A,M) \in \mathcal{K}$.  
\end{proof}

\begin{proof}[Proof of Theorem \ref{thm:strong duality}]
    Theorem \ref{thm:strong duality} follows from Lemmas \ref{lem: delta+ leq gamma+}, \ref{lem: delta+ geq gamma+}, and \ref{lem: K closed}.
\end{proof}

\subsection{Adjoint Functionals}

Let $T \in \sym^n$. We define
\begin{align}
    \rho_1(T) &:= \sup\set{\tr(T\cdot xx^*) = \jbrac{Tx,x} : \norm{x}_1 \leq 1}, \text{ and }
    \label{eq:def:rho+}
    \\
    \pi_+(T) &:= \sup\set{\tr(TA) : A \in \psd^n \text{ and } \norm{A}_{1,1} \leq 1}
    \label{eq:def:pi+}
\end{align}

The suprema are in fact maxima since the feasible regions are compact. We have
\begin{equation*}
    0 \leq \rho_1(T) \leq \pi_+(T) \leq \norm{T}_{\infty,\infty}.
\end{equation*}
See the discussion in Section \ref{sect:psd-relaxation} below.

% Note that the constraint $\norm{x}_1 = 1$ in \eqref{eq:def:rho+} should be interpreted as
% \begin{equation*}
%     1 = \norm{x}_1 = \norm{x}_1^2 = \gp(xx^*) = \g(xx^*).
% \end{equation*}

The name "adjoint" stems from adjoint Banach ideals. See \cite{DJT95}. In \eqref{eq:def:rho+}, if we replace $\jbrac{Tx,x}$ by $\abs{\jbrac{Tx,x}}$, then the right-hand-side is comparable to the operator norm of $T: \ell_1^n \to \ell_\infty^n$, and $(\L(\ell_1^n,\ell_\infty^n),\norm{\cdot}_{1\to\infty})$ is the adjoint ideal of $(\L(\ell_\infty^n,\ell_1^n),\gamma)$. A similar analogy can be made for \eqref{eq:def:pi+}.

Our final reduction of Problem \ref{prob:matrix} is captured by the following. 

\begin{theorem}
\label{thm:equivalence:pi-gamma}
    Given $n \in \mathbb{N}$, we have
    \begin{equation*}
       C_n =\sup_{A \in \psd^n\setminus\set{0}} \frac{\gp(A)}{\norm{A}_{1,1}} = \sup_{T \in \sym^n\setminus\set{0}}\frac{\pi_+(T)}{\rho_1(T)}.
    \end{equation*}
    In the above, we use the convention $\frac{0}{0} = 0$.
\end{theorem}

\begin{proof}
    Note that $\jbrac{z,\frac{T}{\rho_1(T)}z} \leq 1$ for all $T \in \sym^n\setminus\set{0}$, so 
    \begin{equation*}
        \begin{split}
            \sup_{T \in \sym^n\setminus\set{0}}\frac{\pi_+(T)}{\rho_1(T)}
            &=
            \sup_{T\in\sym^n\setminus\set{0},\, \rho_1(T) > 0}\pi_+(T/\rho_1(T))\\
            &= \sup\set{\tr(A\Tilde{T})
            : A \in \psd^n,\, \norm{A}_{1,1} \leq 1,\, \jbrac{\Tilde{T}x,x} \leq 1\,\forall x \in \Sone
            }\\
            \text{(Theorem \ref{thm:strong duality})}\quad&= \sup_{A\in\psd^n\setminus\set{0}}\gp(A/\norm{A}_{1,1}) \\
            \text{(Proposition \ref{prop:gamma+ basics}(A))}\quad &= \sup_{A \in \psd^n\setminus\set{0}} \frac{\gp(A)}{\norm{A}_{1,1}}.
        \end{split}
    \end{equation*}
\end{proof}

\section{Convex relaxation of an $\ell_1$ optimization problem}

\label{sect:psd-relaxation}

In this section, we present an application of Theorem \ref{thm:equivalence:pi-gamma} to a non-convex optimization problem.

Recall the (finite-dimensional) optimization problems associated with $\rho_1$ and $\pi_+$ in \eqref{eq:def:rho+} and \eqref{eq:def:pi+}. 
% By the definition \eqref{eq:def:gamma} of $\g$, we can write $\pi_+(T)$ as
% \begin{equation*}
%     \pi_+(T) = \sup\set{
%     \sum_{k} \jbrac{Tx_k, y_k} : \sum_{k}\norm{x_k}_1\norm{y_k}_1 \leq 1
%     }.
% \end{equation*}
We can then view $\pi_+$ as a \emph{convex relaxation} of $\rho_1$ in the following sense. 
Recall Proposition \ref{prop:special-case-1}(A) and write $\g(A)$ in place of $\g(A:\ell_\infty^n\to \ell_1^n)$.
From $\rho_1$ to $\pi_+$, the constraint
\begin{equation*}
    \norm{x}_1 \leq 1
    \quad \Leftrightarrow \quad \norm{x}_1^2 = \gp(xx^*) = \g(xx^*)  =\norm{xx^*}_{1,1} \leq 1
\end{equation*}
is replaced by a convex higher-rank analogue 
\begin{equation}
   \norm{A}_{1,1} = \g(A) \leq 1.
   \label{eq:A11leq1}
\end{equation}
In the real case, i.e., $A = (A_{kl}) \in \psd^n(\R)$, we can introduce slack variables $(B_{kl})_{k,l = 1}^n$ and rewrite
the constraint \eqref{eq:A11leq1} as
\begin{equation*}
    -B_{kl} \leq A_{kl} \leq B_{kl},
    \,
    B_{kl} \geq 0,\,\text{and }
    \sum_{k,l = 1}^n B_{kl} \leq 1.
\end{equation*}

We obtain the following bound with $C_n$ as in Theorem \ref{thm:equivalence:pi-gamma}.

\begin{theorem}\label{thm:5.1}
For all $T \in \sym^n$, we have
\begin{equation} \label{eq:dualform}
    \rho_1(T) \leq \pi_+(T) \leq C_n\rho_1(T).
\end{equation}
\end{theorem}

In view of Remark \ref{rem:BMS24}, we have $C_n \geq c\sqrt{n}$ for some universal constant $c$.
In contrast with the relaxation of $\norm{\cdot}_{\infty\to 1}$ via (symmetric) Grothendieck's inequality \cite{friedland2020-Grothsym}, the quality of the $(\rho_1,\pi_+)$ relaxation deteriorates at the rate proportional to at least the square root of the ambient dimension for certain matrices.

\begin{remark}
        We know two cases where computing $\rho_1(T)$, $T \in \sym^n$, is easy. 
    \begin{itemize}
        \item Suppose $T$ is negative-semidefinite, then $\rho_1(T) = \pi_+(T) = 0$. 
        \item Suppose $T$ is positive-semidefinite, then
        \begin{equation}
            \eqindent
            \rho_1(T) = \norm{T}_{\infty, \infty} = \max\limits_{k}T(k,k).
            \label{eq:rho+(T)-psd}
        \end{equation}
        The first identity in \eqref{eq:rho+(T)-psd} follows because
        \begin{equation*}
            \sup_{\norm{x}_1 \leq 1}\jbrac{Tx,x} = \sup_{\norm{x}_1 \leq 1}\abs{\jbrac{Tx,x}} = \norm{T}_{\infty,\infty},
        \end{equation*}
        The second identity in \eqref{eq:rho+(T)-psd} follows from the fact that the largest entry of a positive-semidefinite matrix must occur on the diagonal. In this case, $\pi_+(T) = \rho_1(T)$. See Theorem \ref{thm:pi+dual} below. 
    \end{itemize}

\end{remark}

We also provide an alternative formula for $\pi_+$, which can also be interpreted as a strong duality result. 

\begin{theorem}\label{thm:pi+dual}
    For all $T \in \sym^n$, $\pi_+(T) = \min\set{\norm{T+Y}_{\infty,\infty} : Y \in \psd^n}$.
\end{theorem}

A key ingredient in the proof is the following generalization of von Neumann's minmax theorem.

\begin{theorem}[Sion's Minmax Theorem \cite{Sion}]\label{thm:Sion-minmax}
    Let $X$ be a compact convex subset of a topological vector space and let $Y$ be a convex subset of another topological vector space. Suppose $f:X\times Y \to \R$ satisfies the following.
    \begin{itemize}
        \item For all $x \in X$, $f(x,\cdot)$ is upper semicontinuous and quasi-concave on $Y$.
        \item For all $y \in Y$, $f(\cdot, y)$ is lower semicontinuous and and quasi-convex on $X$.
    \end{itemize}
    Then 
    \begin{equation*}
        \min_{x \in X}\sup_{y \in Y} f(x,y) = \sup_{y \in Y}\min_{x \in X} f(x,y).
    \end{equation*}
\end{theorem}

\begin{proof}[Proof of Theorem \ref{thm:pi+dual}]
    Note that $\norm{T+Y}_{\infty,\infty} = \norm{T+Y}_{1\to\infty}$. By $(\ell_\infty^n,\ell_1^n)$ duality\footnote{One can also establish that the adjoint ideal of (self-adjoint) $1$-nuclear operators is the space of (self-adjoint) bounded operators. See \cite{DJT95}.},
    we have
    \begin{equation*}
        \norm{T+Y}_{\infty,\infty} = \max\set{\tr((T+Y)A) : A \in \sym^n,\,  \norm{A}_{1,1} \leq 1}
    \end{equation*}

    The function $f_T: (Y,A) \mapsto \tr((T+Y)A)$ is affine on $\sym^n\times \sym^n$, the set $\set{A \in \sym^n : \norm{A}_{1,1} \leq 1}$ is compact and convex, and the set $\set{Y: Y \in \psd^n}$ is convex. By Sion's Theorem \ref{thm:Sion-minmax} and the fact that $\inf (-f_T) = \sup f_T$, we have
    \begin{equation*}
        \begin{split}
            \inf_{Y \in \psd^n} \max_{\norm{A}_{1,1} \leq 1}\tr((T+Y)A)
            &= \max_{\norm{A}_{1,1} \leq 1}\inf_{Y \in \psd^n} \tr((T+Y)A) 
            \\
            &= \max_{\norm{A}_{1,1}\leq 1} \brac{
            \tr(TA) + \inf_{Y\in\psd^n}\tr(YA)
            }
            \\
            &= \max_{\norm{A}_{1,1}\leq 1,\, A \in \psd^n}\brac{\tr(TA) + \inf_{Y\in\psd^n}\tr(YA)}\\
            &= \max_{\norm{A}_{1,1}\leq 1,\, A \in \psd^n}\tr(TA) = \pi_+(T).
        \end{split}
    \end{equation*} 

    To see that the infimum can be replaced by the minimum, we fix $T$ and let $Y_j$ be a sequence in $\psd^n$ such that $\norm{T+Y_j}_{\infty,\infty} \to \pi_+(T)$ as $j \to \infty$.
    Since $Y_j$ is bounded in $\norm{\cdot}_{\infty,\infty}$, we may extract a subsequence, still denoted by $Y_j$, such that $Y_j \to Y_0$ entry-wise, and $\norm{T+Y_0}_{\infty,\infty} = \pi_+(T)$. Since $\psd^n$ is closed, we have $Y_0 \in \psd^n$.
\end{proof}

\section{The infinite-dimensional problem}
\label{sect:infinite-dim problem}

Recall that $c_0$ denotes the vector space of sequences that converge to 0. We prove a proposition that justifies the abuse of notation for the rest of this paper.  

\begin{proposition}\label{prop:infinity-->one}
    Let $A \in \L(\ell_2,\ell_2)$ satisfy $\norm{A}_{1,1} < \infty$. Then the induced action
    \begin{equation}
        \widetilde{A}: b\mapsto \sum_{k = 1}^{\infty} b(k)A\delta_k
        \label{eq: induced action}
    \end{equation}
    defines a bounded operator $\ell_\infty \to \ell_q$ for all $q \in [1,\infty)$ with $\|\widetilde{A}\|_{\infty \to q} \leq \norm{A}_{1,1}$. In addition, $\widetilde{A} : c_0 \to \ell_q$ is nuclear for all $q \in [1,\infty)$.
\end{proposition}

\begin{proof}
    Let $b \in \ell_\infty$. If $q = 1$, we can directly estimate 
    \begin{equation*}
        \norm{\sum_k b(k)A\delta_k}_1 \leq \sum_k \abs{b(k)}\norm{A\delta_k}_1 = \sum_k \abs{b(k)}\sum_l\abs{\jbrac{A\delta_k,\delta_l}} \leq \norm{A}_{1,1}\norm{b}_{\infty}.
    \end{equation*}
    Suppose $q > 1$. Let $x \in \ell_q^* \simeq \ell_{q^*}$ where $1/q + 1/q^* = 1$. We can write $x = \sum_{k} x(k)\delta_k$ and estimate
    \begin{equation*}
        \begin{split}
            \abs{
            \jbrac{\sum_k b(k)A\delta_k, x}
            }
            &= 
            \abs{\jbrac{\sum_k A\delta_k, \sum_l x(l)\delta_l}}
            \leq \sum_{k,l}\abs{b(k)x(l)}\abs{\jbrac{A\delta_k,\delta_l}} \leq \norm{A}_{1,1}\norm{b}_{\infty}\norm{x}_{q^*}.
        \end{split}
    \end{equation*}
    Since $x$ is arbitrary, we can conclude that $\sum_k b(k)A\delta_k \in \ell_q$, and $\|\widetilde{A}\|_{\infty \to q} \leq \norm{A}_{1,1} < \infty$.

    Finally, the assumption $\norm{A}_{1,1} < \infty$ also implies $\norm{A}_{q,1} < \infty$ for all $q \in [1,\infty)$. Thanks to Proposition \ref{prop:nuclear-vs-entry} and the argument above, we can conclude that $\widetilde{A} : c_0 \to \ell_q$ is nuclear for all $q \in [1,\infty)$. 
\end{proof}

\textbf{Convention.} For simplicity, we will abuse notation and write $A$ instead of $\widetilde{A}$ for the rest of this paper without further referring to the action \eqref{eq: induced action} (which is nothing more than the standard matrix multiplication in the finite-dimensional setting).

We now define the key objects of this section:
\begin{equation}
    \begin{split}
        \V &:= \set{A \in \L(\ell_2, \ell_2) : A = A^*,\, \norm{A}_{1,1} = \sum_{k,l = 1}^{\infty}\abs{\jbrac{A\delta_k,\delta_l}} < \infty} \text{ and }\\
        \Gamma:&= \set{A \in \V : A\succ 0 \text{ as a bilinear form on $\ell_2$}}.
    \end{split}
    \label{eq:def:Gamma}
\end{equation}
By Proposition \ref{prop:infinity-->one}, $\V$ is the Banach space of all nuclear operators $c_0 \to \ell_1$ that are self-adjoint as a bilinear form on $\ell_2$ (or $c_0$), and $\Gamma$ is the positive semidefinite cone in $\V$. 

We define $\gp$ accordingly: For $A \in \Gamma$, we set
\begin{equation*}
    \gp(A) := \inf\set{\sum_{k = 1}^\infty \norm{x_k}_1^2: \, A = \sum_{k = 1}^{\infty} x_kx_k^* \text{ strongly }}
\end{equation*}
with the convention that $\inf \void = \infty$ and we define
\begin{equation}
    \Gamma_+:= \set{A \in \Gamma: \gp(A) < \infty}.
    \label{eq:def:Gamma+}
\end{equation}

Note that the condition $\sum_{k}\norm{x_k}_1^2 < \infty$ implies that the convergence of the series $\sum_{k= 1}^{N}x_kx_k^*$ to $A$ can be with respect to the nuclear topology on $c_0 \to \ell_1$, and hence, every other coarser topology such as the operator norm topology. This is because, in view of Proposition \ref{prop:nuclear-vs-entry},
    \begin{equation}
        \g\brac{A - \sum_{k = 1}^Nx_kx_k^*} \leq \gp\brac{A - \sum_{k = 1}^N x_kx_k^*} \leq \sum_{k = N+1}^{\infty} \norm{x_k}_1^2 \to 0
        \text{ as } N \to \infty.
    \end{equation}

\begin{remark}
    Heil and Larson \cite{HeilLarson08} refer to operators from $\Gamma_+$ as {\em Type B} with respect to the canonical orthonormal $\ell_2$ basis $\set{\delta_k}$.  
\end{remark}
\begin{remark}\label{rem:gamma+subadditive}
    As in the finite-dimensional setting, $\gp$ is subadditive and positive-homogeneous on $\Gamma_+$, namely, $\gp(A+B)\leq \gp(A) + \gp(B)$ and $\gp(cA) = c \cdot \gp(A)$ for all $A,B \in \Gamma_+$ and $c \geq 0$.
\end{remark}

\subsection{A negative answer to Problem \ref{prob:Feichtinger}}

\begin{remark}\label{rem:BMS24}
    In \cite{BMS24}, the authors showed that 
there exists some universal constant $c > 0$ such that the following holds. 
    Given $n \in \mathbb{N}$, there exists a matrix $T_n \in \sym^n$ such that
    \begin{equation*}
        \frac{\pi_+(T_n)}{\rho_1(T_n)} \geq c\sqrt{n}. 
    \end{equation*}
\end{remark}
As a consequence, we have the following theorem.

\begin{theorem}\label{thm:Feichtinger->matrix}
   There exists $A \in \Gamma$ but $A \notin \Gamma_+$.
\end{theorem}

\begin{proof}
    By Theorem \ref{thm:equivalence:pi-gamma} and Remark \ref{rem:BMS24}, for every $n$, there exists a matrix $A_n \in \psd^n$ with 
    \begin{equation*}
        \norm{A_n}_{1,1} = 1
        \text{ but }
        \gp(A_n) \geq \frac{c\sqrt{n}}{2}
    \end{equation*}
    with $c$ as in Remark \ref{rem:BMS24}.
    
    Consider the infinite block-diagonal matrix defined by
    \begin{equation*}
        A:= \bigoplus_{n = 1}^\infty n^{-3/2}A_n.
    \end{equation*}
    It follows that $A \in \Gamma$. 
    
   For $n \in \mathbb{N}$, let $P_n $ denote the orthogonal projection onto the range of $A_n$. We may then write
    \begin{equation*}
        A = \sum_{m,n = 1}^\infty P_m A P_n = \sum_{m,n = 1}^\infty \mathbf{1}_{m=n}\cdot P_mAP_n = \sum_{n = 1}^\infty P_n A P_n,
    \end{equation*}
    convergent in the strong topology. 
    
     Let $(h_k) \subset \ell_1$ be any decomposition of $A = \sum_k h_kh_k^*$ in the strong topology. Then
     \begin{equation*}
         A = \sum_{n=1}^\infty P_n\brac{\sum_{k = 1}^\infty h_k h_k^* } P_n = \sum_{n,k=1}^\infty (P_nh_k)(P_nh_k)^*
     \end{equation*}
     and
     \begin{equation*}
         n^{-3/2}A_n = \sum_{k = 1}^\infty (P_nh_k)(P_nh_k)^*.
     \end{equation*}
    Note that
    \begin{equation*}
        \sum_{n = 1}^\infty\norm{P_nh_k}_1^2 = \sum_{n = 1}^\infty \sum_{i,j = 1}^\infty \abs{(P_nh_k)(i)}\abs{(P_nh_k)(i)} \leq 
        \sum_{i,j = 1}^\infty \abs{h_k(i)}\abs{h_k(j)} = 
        \norm{h_k}_1^2.
    \end{equation*}
    As a consequence,
    \begin{equation*}
        \sum_{k=1}^\infty \norm{h_k}_1^2 \geq \sum_{n,k = 1}^\infty \norm{P_nh_k}_1^2 \geq \sum_{n=1}^\infty\frac{1}{n^{3/2}}\gp(A_n) \geq 
        \frac{c}{2}\sum_{n = 1}^\infty \frac{1}{n} = \infty.
    \end{equation*}
\end{proof}

\begin{remark}
    As seen in the proof, the counterexample in Theorem \ref{thm:Feichtinger->matrix} can be made explicit if the constant $C_n$ in \eqref{eq: Cn BMS} can be explicitly realized (or approximately realized) by a sequence of matrices $A_n$ for all large $n$. This shall not be taken for granted, however, as the construction in \cite{BMS24} is given in the adjoint (Remark \ref{rem:BMS24}). It is not clear to the authors how to find such $A_n$'s.
\end{remark}

\subsection{Topology and Density}
\label{sect:topology and density}
\newcommand{\U}{\mathcal{U}}

Let $\tau$ be the topology on $\V$ generated by open sets of the form $T + \U_\eps$ for $\eps > 0$,
where $T \in \V$ and 
\begin{equation}
    \U_\eps := \set{A \in \Gamma_+ : \gp(A) < \eps}.
    \label{eq:tau-topology-U_eps}
\end{equation}
Intuitively, $\tau$ measures ``approximation from above''.
In view of Theorem \ref{thm:Feichtinger->matrix}, we see that
\begin{equation*}
    \U_\eps \subsetneq \set{A \in \Gamma_+: \norm{A}_{1,1} < \eps} \subsetneq \set{A \in \V: \norm{A}_{1,1} < \eps}.
\end{equation*}

\begin{proposition}\label{prop:topology, basis, non-second countability}
    The following are true about $\tau$.
    \begin{enumerate}[(A)]
        \item $\tau$ is Hausdorff, translation invariant, and first-countable.
        \item The family $\mathcal{F} = \set{T+ \U_\eps: T \in \V, \eps > 0}$ forms a basis for $\tau$, i.e., every open set in $\tau$ can be written as a union of elements in $\mathcal{F}$.
        \item If $A\in \mathcal O\in\tau$ then there exists $r>0$ so that $A+\U_r\subset \mathcal O$. In particular,   $\tau$ is not the discrete topology on $\V$.
        \item Let $A_0\in\Gamma\setminus\Gamma_+$ (whose existence is guaranteed by Theorem \ref{thm:Feichtinger->matrix}). Then the subspace $\{tA_0;t\in\R\}$ with the induced topology is discrete.
        \item $\tau$ is not second-countable.
    \end{enumerate}
\end{proposition}

\begin{proof}
    For (A), it is clear that $\tau$ is Hausdorff and translation-invariant. The family $\set{\U_{q}:q \in \mathbb{Q},q>0}$  
    gives a local countable base at $0 \in \V$, so $\tau$ is first-countable. 
    
    We turn to (B). Since $\tau$ is generated by $\mathcal{F}$, it suffices to show that the intersection of two open sets in $\mathcal{F}$ can be written as a union of open sets in $\mathcal{F}$. To this end, suppose $T_i + \U_{\eps_i} \in \mathcal{F}$ for $i = 1, 2$ and they have nonempty intersection. Pick any
    \begin{equation*}
        T_0 \in (T_1 + \U_{\eps_1})\cap(T_2 + \U_{\eps_2}).
    \end{equation*}
    This means that for each $i$, there exists $A_i \in \U_{\eps_i}$ such that $T_0 = T_i + A_i$. Set
    \begin{equation*}
        \eps_0 = \frac{1}{2}\min\set{\eps_i - \gp(A_i)}.
    \end{equation*}
    For any $E \in \U_{\eps_0}$, we have
    \begin{equation*}
        \gp(A_i + E)\leq \gp(A_i)+\gp(E) < \gp(A_i) + \eps_0 < \eps_i,
    \end{equation*}
    so $A_i + \U_{\eps_0}\subset \U_{\eps_i}$ for each $i$. Therefore,
    \begin{equation*}
        T_0 + \U_{\eps_0} \subset (T_0 - A_i)+ \U_{\eps_i} = T_i + \U_{\eps_i}
        \text{\quad for $i=1,2$.}
    \end{equation*}
    The conclusion follows since $T_0$ is arbitrary. 

    We turn to Part (C). Thanks to part (B), the open set $\mathcal O$ can be written as $\mathcal O=\cup_{i\in I}(B_i+\U_{r_i})$ for some index set $I$, where $B_i\in\V$ and $r_i>0$. Since $A\in \mathcal \mathcal O$, there is $i_0\in I$ so that $A\in B_{i_0}+\U_{r_{i_0}}\subset \mathcal \mathcal O$. Thus $A=B_{i_0}+C$ for some $C\in\Gamma_+$, with $\gamma_+(C)<r_{i_0}$. Let $r:=r_{i_0}-\gamma_+(C)>0$. We claim that $A+\U_r\subset B_{i_0}+\U_{r_{i_0}}$, and hence, $A+\U_r\subset \mathcal \mathcal O$.  
    To see this, let $E\in \U_{r}$. We have $\gamma_+(E)<r$ and $\gamma_+(C+E)\leq \gamma_+(C)+\gamma_+(E)<r_{i_0}$. Thus $A+E=B_{i_0}+C+E\in B_{i_0}+\U_{r_{i_0}}$. Therefore, singletons are not open sets, $\{A\}\not\in\tau$, and $\tau$ is not the discrete topology on $\V$.

    We now turn to (D) and (E). Thanks to Theorem \ref{thm:Feichtinger->matrix}, we can find and fix
    \begin{equation}
        A_0 \in \Gamma\setminus \Gamma_+ \text{, i.e., $A_0 \in \V$ but $\gp(A_0) = \infty$.}
        \label{eq:second-countable bad element}
    \end{equation}
    For any $s \neq t \in \R$, $tA_0 \notin sA_0 + \U_\eps$ for any $\eps$, since either $t-s < 0$ or
    \begin{equation*}
        \gp((t-s)A_0) = (t-s)\cdot\gp(A_0) = (t-s)\cdot \infty = \infty
    \end{equation*}
    by \eqref{eq:second-countable bad element}. 
    Consider the uncountable set $S = \set{tA_0 : t \in \R}\subset \V$ equipped with the subspace topology induced by $\tau$. By the argument above, for each $tA_0 \in S$, there exists a basic open set $tA_0 + \U_1 \in \mathcal{F}$ containing $tA_0$ such that
    \begin{equation*}
        (tA_0 + \U_1)\cap S = \set{tA_0},
    \end{equation*}
    i.e., $tA_0$ is isolated in the subspace topology. Thus,
    $S$ discrete in the subspace topology and not second-countable. Since every subspace of a second-countable space must be second-countable, $\tau$ is not second-countable.
    
\end{proof}

\begin{remark}
    The space $(\V,\tau)$ is {\em not} a topological vector space. There is no balanced neighborhood of zero, and scalar multiplication is not continuous at zero either. 
\end{remark}

\begin{theorem}\label{thm:density}
    Let $\V$, $\Gamma$, and $\Gamma_+$ be as in \eqref{eq:def:Gamma} and \eqref{eq:def:Gamma+}. Then
    \begin{enumerate}[(A)]
        \item $\Gamma_+$ is dense in $\Gamma$ with respect to $\tau$,
        \item $\Gamma$ is the the $\tau$-closure of $\Gamma_+$, and
        \item $\Gamma_+$ is reproducing, namely, $\V = \Gamma_+ - \Gamma_+$.  
    \end{enumerate}
    
\end{theorem}

\begin{proof}
    For Part (A), it suffices to show that for any $A \in \Gamma$ and $\eps > 0$, we can find operators $B,C \in \Gamma_+$ with $A = B-C$ and $\gp(C) \leq \eps$.
    To this end, fix $A \in \Gamma$ and $\eps > 0$. If $\norm{A}_{1,1} = 0$, then $A = 0 \in \Gamma\cap \Gamma_+$. Without loss of generality, we may assume that $\norm{A}_{1,1} = 1$ for the rest of the proof. 
    
    Let $N$ be sufficiently large such that
    \begin{equation}
        \sum_{k,l = 1}^{N}\abs{\jbrac{A\delta_k,\delta_l}} \geq 1-\frac{\epsilon}{2}.
        \label{eq:thm:density-1}
    \end{equation}
    Let $P_N:\ell_2 \to \ell_2$ be the orthogonal projection onto the first $N$ entries,
    and define the truncated operator $A_N :\ell_2 \to \ell_2$ by
    \begin{equation*}
        A_Nx = \sum_{k = 1}^{N}x(k)P_NA\delta_k.
    \end{equation*}
    It follows that $A_N \succ 0$ and
    \begin{equation*}
        \norm{A_N}_{1,1},\, \gp(A_N) < \infty.
    \end{equation*}

    Let $R = A - A_N$ be the residual operator. Note that $R$ is self-adjoint but not necessarily positive-semidefinite. It follows from \eqref{eq:thm:density-1} that
    \begin{equation}
        \norm{R}_{1,1} \leq \frac{\epsilon}{2}.
        \label{eq:thm:density-2}
    \end{equation}
    Let $r_k := R\delta_k$. Thanks to Proposition \ref{prop:infinity-->one}, $r_k \in \ell_1 \subset \ell_2$.
    We may then write $R$ as the strongly convergent series
    \begin{equation}
        R =  \sum_{k = 1}^\infty r_k\delta_k^* = \sum_{k = 1}^\infty \delta_kr_k^* =  \sum_{k=1}^{\infty}\frac{1}{2}(r_k\delta_k^* + \delta_kr_k^*) = \sum_{k = 1}^{\infty} f_kf_k^* - g_kg_k^*
        \label{eq:thm:density-3}
    \end{equation}
    where
    \begin{equation}
        f_k = \frac{1}{2}\brac{ \frac{r_k}{\sqrt{\norm{r_k}_1}} + \sqrt{\norm{r_k}_1}\delta_k }
        \text{ and }
        g_k =  \frac{1}{2}\brac{ \frac{r_k}{\sqrt{\norm{r_k}_1}} -\sqrt{\norm{r_k}_1}\delta_k }
        \label{eq:thm:density-4}
    \end{equation}
   It follows from \eqref{eq:thm:density-2}--\eqref{eq:thm:density-4} that
   \begin{equation}
       \sum_{k=1}^{\infty} \norm{f_k}_1^2 + \norm{g_k}_1^2 \leq 2\sum_{k=1}^{\infty}\norm{r_k}_1 = 2\norm{R}_{1,1} \leq \epsilon.
       \label{eq:thm:density-5}
   \end{equation}
   If we set 
   \begin{equation*}
       B = A_N + \sum_{k=1}^{\infty} f_kf_k^*
       \text{ and } C = \sum_{k = 1}^\infty g_kg_k^*,
   \end{equation*}
   we see that $B,C \succ 0$,
   $A = B- C$. Moreover, by Remark \ref{rem:gamma+subadditive}, 
   \begin{equation*}
       \gp(B) \leq \gp(A_N) + \sum_{k = 1}^{\infty}\norm{f_k}_1^2 < \infty.
   \end{equation*}
   Thanks to \eqref{eq:thm:density-5}, we also have
   \begin{equation*}
       \gp(C) \leq \sum_{k = 1}^{\infty}\norm{g_k}_1^2 \leq \epsilon
   \end{equation*}
   This proves Part (A).

   For Part (B), it suffices to show that $\V \setminus \Gamma$ is open in $\tau$. If $S \in \V \setminus \Gamma$, then that there exists $v \in \ell_2$ with $\norm{v}_2 = 1$ and $\jbrac{Sv,v} = -\alpha < 0$. 
   % We note that $v \in \ell_1$, since\begin{equation*}\norm{v}_1 = \alpha^{-1}\norm{\alpha v}_1 \leq \alpha^{-1}\norm{S}_{2\to 1}\norm{v}_2\leq \alpha^{-1}\norm{S}_{1,1} < \infty\end{equation*}
   Let $\eta > 0$ be sufficiently small, such that any $E \in \U_\eta$ (see \eqref{eq:tau-topology-U_eps}) must satisfy $\jbrac{Ew,w} < \frac{\alpha}{2}\norm{w}_2^2$ for all $w \in \ell_2$. It follows that $S + \U_\eta \subset \V\setminus\Gamma$.

   For Part (C), we may invoke Proposition \ref{prop:nuclear-vs-entry} and use the same decomposition technique in \eqref{eq:thm:density-3} to show that any operator $T \in \V $ can be written as $T = W-Z$ with $W,Z \in \Gamma_+$. We omit the details here.
\end{proof}

\begin{remark}
    As a corollary of Theorem \ref{thm:density}(A), we see that $\Gamma_+$ is dense in $\Gamma$ in the $\norm{\cdot}_{1,1}$-topology. However, this statement is more or less trivial, since $\Gamma $ consists of nuclear operators from $c_0 \to \ell_1$ so they are the $\norm{\cdot}_{1,1}$-limit of finite-rank operators, and finite-rank operators always admit a factorization satisfying \eqref{eq:condition-strongl2-intro}.
\end{remark}

\subsection{Action under positive-coefficient Wiener subalgebra}
\label{sect:holomorphic action}

The main result of this section is the following.

\begin{theorem}\label{thm:holomorphic action}
    Let $f(z) = \sum_{k=1}^\infty c_kz^k $ with $c_k \geq 0$ for all $k$ and $\sum_{k=1}^\infty c_k < \infty $. Let $g(z) = \tfrac{f(z)}{z} = \sum_{k=0}^\infty c_{k+1}z^k$. Then for every $A \in \Gamma_+$, we have $g(A) \in \Gamma_+$ with
    \begin{equation}
        \gp(f(A)) \leq g(\norm{A}_{1,1}) \cdot \gp(A).
        \label{eq:holomorphic-0}
    \end{equation}
\end{theorem}

We will first present two key ingredients, Lemmas \ref{lem:square root lp to l1} and \ref{lem:L1 endo}, which are of independent interest.

The celebrated Grothendieck-Pietsch factorization theorem implies that that any operator $A \in \L(\ell_\infty,\ell_1)$ factors through $\ell_2$ (see, e.g., \cite[Corollary 2.2]{pisier2012grothendieck}). The following lemma (with $p = 2$) can be viewed as an explicit variant of this classical result given the additional assumption that $A$ is self-adjoint and positive.

\begin{lemma}\label{lem:square root lp to l1}
    Let $A \in \Gamma$ and let $A^{1/2}:\ell_2 \to \ell_2$ be the square root of $A : \ell_2 \to \ell_2$. Then for any $1 \leq p \leq 2$, $A^{1/2} \in \L(\ell_p, \ell_1)$ with 
    \begin{equation*}
     \norm{A^{1/2}}_{p\to 1} \leq \norm{A}_{1,1}^{1/2}. 
    \end{equation*}
    In particular, we have $A^{1/2}:\ell_2 \to \ell_1$, and by self-adjointness, $A^{1/2}: \ell_\infty \to \ell_2$, so any $A\in \Gamma$ factors through $\ell_2$.
\end{lemma}

\begin{proof}

Let $c_{00}$ denote the space of sequences with finite support. For any $z \in \ell_1$, we have
\begin{equation*}
    \norm{z}_1 = \sup\limits_{w \in c_{00}, \norm{w}_\infty \leq 1}\jbrac{z,w}.
\end{equation*}
% Also note that $c_{00}$ is $\norm{\cdot}_p$-dense in $\ell_p$ for all $1 \leq p < \infty$.
Let $x \in \ell_p \subset \ell_2$ and let $y \in c_{00} \subset \ell_2$. We have
\begin{equation*}
    \abs{\jbrac{A^{1/2}x,y}} = \abs{\jbrac{x,A^{1/2}y}} \leq \norm{x}_2\norm{A^{1/2}y}_2 = \norm{x}_2\sqrt{\jbrac{A^{1/2}y, A^{1/2}y}} = \norm{x}_2\jbrac{Ay,y}^{1/2}.
\end{equation*}
On the other hand,
\begin{equation*}
    \jbrac{Ay,y} \leq \sum_{i,j=1}^{\infty}\abs{A(i,j)}\abs{y(i)}\abs{y(j)}
    \leq \norm{A}_{1,1}\norm{y}_{\infty}^2. 
\end{equation*}
Combining the two estimates above and using $\norm{x}_2 \leq \norm{x}_p$ whenever $p\in [1,2]$, we have
\begin{equation*}
    \abs{\jbrac{A^{1/2}x, y}} \leq \norm{A}_{1,1}^{1/2}\norm{x}_p\norm{y}_\infty.
\end{equation*}
% The conclusion follows since $x$ and $y$ are arbitrary.
Therefore,
\begin{equation*}
    \norm{A^{1/2}}_{p\to 1} = \sup_{x \in \ell_p, \norm{x}_p \leq 1} \sup_{y \in c_{00}, \norm{y}_\infty \leq 1}\abs{\jbrac{A^{1/2}x,y}} \leq \norm{A}_{1,1}^{1/2}.
\end{equation*}
\end{proof}

The next lemma states that $\Gamma_+$ is invariant under the quadratic action of $\ell_1$ endomorphisms.

\begin{lemma}\label{lem:L1 endo}
    If $T \in \L(\ell_1,\ell_1)$ and $A \in \Gamma_+$, then
    Then $TAT^* \in \Gamma_+$ with
    \begin{equation}
    	\gamma_+(TAT^*) \leq \gp(A)\norm{T}_{1\to 1}^2.
    	\label{eq:L1 endo-0}
    \end{equation} 
\end{lemma}

\begin{proof}
	Let $\eps > 0$ and let $(x_k)\subset \ell_1$ be such that
    \[
    A = \sum_k x_kx_k^*
        \text{ with } \sum_k \norm{x_k}_1^2 \leq \gp(A) + \eps < \infty. 
    \]
    We have a strongly convergent factorization $TAT^* = \sum_k (Tx_k)(Tx_k)^*$
    with the property
    \[
    \sum_k \norm{Tx_k}_1^2 \leq \norm{T}_{1\to 1}^2\sum_k\norm{x_k}_1^2 \leq \gp(A)\norm{T}_{1\to 1}^2 + \eps.
    \]
    The conclusion follows since $\eps > 0$ is arbitrary. 
\end{proof}

\begin{proof}[Proof of Theorem \ref{thm:holomorphic action}]

Recall from Remark \ref{rem:gamma+subadditive} that $\gp$ is positive-homogeneous and sub-additive. Therefore, it suffices to show that 
\begin{equation}
	\gp(A^k) \leq \norm{A}_{1,1}^{k-1}\cdot \gp(A)
    \text{\quad for all $k\geq 1$.}
    \label{eq:holomorphic-induction-0}
\end{equation}
We proceed by induction.

{\em Base case:} when $k = 1$, \eqref{eq:holomorphic-induction-0} is trivial.

{\em Induction hypothesis: } suppose \eqref{eq:holomorphic-induction-0} holds for $1, \cdots, k-1$. We write $A^k = A^{1/2}A^{k-1}A^{1/2}$. By the induction hypothesis,
\begin{equation}
    \gp(A^{k-1})\leq \norm{A}_{1,1}^{k-2}\cdot \gp(A)
    \label{eq:holomorphic-induction-1}
\end{equation}
Using Lemma \ref{lem:L1 endo}, Lemma \ref{lem:square root lp to l1}, and \eqref{eq:holomorphic-induction-1}, we have
\begin{equation*}
    \gp(A^{1/2}A^{k-1}A^{1/2}) \leq \norm{A^{1/2}}_{1\to 1}^2 \cdot \gp(A^{k-1}) \leq \norm{A}_{1,1} \cdot \norm{A}_{1,1}^{k-2}\cdot \gp(A) = \norm{A}_{1,1}^{k-1}\gp(A).
\end{equation*}
This concludes the proof of Theorem \ref{thm:holomorphic action}.
\end{proof}

\subsection{Factorization into $2$-summing operators}
\label{sect:2-summing}

\newcommand{\strong}{{\mathsf{strong}}}
\newcommand{\weak}{{\mathsf{weak}}}

In this section, we provide a sufficient condition for an eigendecomposition of $A \in \Gamma$ to be $\ell_1$ norm-squared summable (i.e., a strong $\ell_2(\ell_1)$ sequence according to the definition to follow). Our criterion is stated in terms of factoring $A$ into a composition of $2$-summing operators (Theorem \ref{thm:2-summing factorization, sufficient Type A}). In addition, we give a necessary and sufficient condition for $A^{1/2}$ to be $2$-summing (Theorem \ref{thm:square root 2-summing, sufficient, square root diagonal}).

Let $1 \leq p < \infty$.
We say a sequence $(x_k)$ in a Banach space $X$ is a strong $\ell_p$ sequence if
\begin{equation*}
    \norm{(x_k)}_{\ell_p^{\strong}(X)}^p := \sum_k \norm{x_k}_X^p < \infty.
\end{equation*}
We say $(x_k)$ is a weak $\ell_p$ sequence if
\begin{equation*}
    \norm{(x_k)}_{\ell_p^{\weak}(X)}^p := \sup_{x^* \in B_{X^*}}
    \sum_k \abs{\jbrac{x^*,x_k}}^p
     < \infty.
\end{equation*}
We use $\ell_p^\strong(X)$ (and respectively, $\ell_p^\weak(X)$) to denote the Banach space of strong (and respectively, weak) $\ell_p$ sequences in $X$.

Let $X,Y$ be Banach spaces. 
A linear operator $T: X \to Y$ is called $p$-summing ($1 \leq p < \infty$) if there exists a constant $C \geq 0$ such that for any finite collection $x_1, \cdots, x_N \in X$, we have
\begin{equation}
    \sum_{k=1}^N \norm{Tx_k}_Y^p \leq C^p \cdot \sup_{x^* \in B_{X^*}}
    \sum_{k=1}^N \abs{\jbrac{x^*,x_k}}^p.
    \label{eq:p-summing def}
\end{equation}
The infimum of all $C$ for which \eqref{eq:p-summing def} always holds is denoted by $\pi_p(T) = \pi_p(T:X\to Y)$, called the $p$-summing norm of $T$. By choosing a single vector in the definition \eqref{eq:p-summing def}, we see that 
\begin{equation*}
    \norm{T}_{X\to Y} \leq \pi_p(T)
    \text{\quad for any $p \in [1,\infty)$.}
\end{equation*}
We use $\Pi_p(X,Y)$ to denote the vector space of $p$-summing operators $X \to Y$ with equipped with the norm $\pi_p(\cdot)$, under which $\Pi_p(X,Y)$ is a Banach space (moreover, a Banach ideal). Alternatively, $T$ is $p$-summing if and only if the following induced action defines a bounded linear operator 
\begin{equation}
    (x_k)_{k=1}^\infty \mapsto (Tx_k)_{k=1}^\infty \quad : \quad 
    \ell_p^\weak(X) \to \ell_p^\strong(Y)
    \label{eq:p-summing def alt}
\end{equation}
and the norm of this operator is precisely $\pi_p(T)$. 

Note that $p$-summing operators are $q$-summing whenever $1 \leq p \leq q < \infty$. If $X$ and $Y$ are both Hilbert spaces, then the $2$-summing norm coincides with the Hilbert-Schmidt norm.

\begin{theorem}\label{thm:2-summing factorization, sufficient Type A}
    If $A\in \Gamma$ admits a factorization $A = TT^*$ such that $T\in\Pi_2(\ell_2,\ell_1)$, then $A \in \Gamma_+$. Moreover, for any eigendecomposition $A = \sum_k \lambda_k e_ke_k^* $, we have
    \begin{equation*}
        \gp(A) \leq 
        \sum_{k}\norm{A^{1/2}e_k}_1^2=
        \sum_k \lambda_k\norm{e_k}_1^2 \leq \pi_2(T)^2 < \infty.
    \end{equation*}
\end{theorem}

\begin{proof}
    Let $I$ denote the identity on $\ell_2$.
    Let
    $(f_k)\subset \ell_2$ be a tight frame of $\ell_2$ with frame bound $\alpha > 0$, i.e., $I = \alpha^{-1}\sum_k f_kf_k^*$. We then have
    \begin{equation*}
        A = TT^* = TIT^* = \alpha^{-1}T\brac{\sum_k f_kf_k^*}T^* = \sum_k (\alpha^{-1/2}Tf_k)(\alpha^{-1/2}Tf_k)^*
    \end{equation*}
    with
    \begin{equation*}
        \sum_k \norm{\alpha^{-1/2}Tf_k}_1^2 \leq \pi_2(T)^2\cdot\alpha^{-1} \sup_{\norm{\phi}_2\leq 1}\sum_k\abs{\ip{\phi,f_k}}^2 \leq \pi_2(T)^2.
    \end{equation*}
    This means that $A \in \Gamma_+$.
    
    For the second part of the theorem, we may assume, without loss of generality, that $\lambda_k > 0$ for all $k$. We set
    \begin{equation*}
        u_k := \lambda_k^{-1/2}T^*e_k.
    \end{equation*}
    Notice that
    \begin{align}
        \ip{u_k,u_l} &= \lambda_k^{-1/2}\lambda_l^{-1/2}\ip{TT^*e_k,e_l} = \lambda_k^{1/2}\lambda_l^{-1/2}\ip{e_k,e_l} = \begin{cases}
            1 &\text{ if } k = l\\
            0 &\text{otherwise}
        \end{cases}\text{, and }
        \label{eq:2-summing, sufficient - 1}
        \\
        \norm{Tu_k}_1^2 &= \norm{\lambda_k^{-1/2}Ae_k}_1^2 = \lambda_k\norm{e_k}_1^2.\label{eq:2-summing, sufficient - 2}
    \end{align}
    Using \eqref{eq:2-summing, sufficient - 1}, \eqref{eq:2-summing, sufficient - 2}, and the alternative characterization \eqref{eq:p-summing def alt} of the $2$-summability of $T$, we have
    \begin{equation*}
        \gp(A) \leq \sum_k \lambda_k\norm{e_k}_1^2 = \sum_k \norm{Tu_k}_1^2 \leq \pi_2(T)^2 \cdot \sup_{\norm{\phi}_2 \leq 1}\sum_k \abs{\ip{\phi,u_k}}^2 \leq \pi_2(T)^2.
    \end{equation*}

\end{proof}

\begin{remark}
    In the language of Heil and Larson \cite{HeilLarson08}, Theorem \ref{thm:2-summing factorization, sufficient Type A} says that if an element in $\Gamma_+$ (i.e., an operator of Type B) admits a $2$-summing factorization as in the hypothesis, then it is of Type A  with respect to the canonical orthonormal $\ell_2$ basis $\set{\delta_k}$.  
\end{remark}

Recall the following version of the celebrated Grothendieck inequality (see, e.g., \cite[Theorems 3.7 and 3.11]{DJT95}).

\begin{theorem}[Grothendieck inequality]\label{thm:Grothendieck}
    $\L(\ell_\infty,\ell_2)\subset \Pi_2(\ell_\infty,\ell_2)$ with $\pi_2(T) \leq \kappa_G\norm{T}_{\infty\to 2}$, where $\kappa_G$ is the Grothendieck constant. As a consequence, $\Pi_2(\ell_2,\ell_1) = \Pi_1(\ell_2,\ell_1)$ with $\pi_2(T)\leq \pi_1(T)\leq \kappa_G\cdot \pi_2(T)$ for all $T:\ell_2 \to \ell_1$.
\end{theorem}

Using the Hilbert space structure of $\ell_2$, Paneque and Pi\~neiro \cite[Theorem 3]{paneque2000banach} showed that $\Pi_2(\ell_2,\ell_1)$ coincides with the space of nuclear operators $T:\ell_2 \to \ell_1$ and 
\begin{equation}
    \begin{split}
        \pi_2(T) &\approx \g(T:\ell_2\to \ell_1)
        \\
        &= \inf\set{\sum_k\norm{x_k^*}_2\norm{y_k}_1 : 
        T = \sum_kx_k^*\otimes y_k \text{ strongly}
        } = \norm{T^*}_{2,1}
    \end{split}
    \label{eq:P&P}
\end{equation}
for all $T:\ell_2 \to \ell_1$ up to a universal constant of equivalence. We reprove a special case of their result with explicit constants. 

\begin{proposition}\label{prop:T* nuclear implies T 2-summing}
    Let $T\in \L(\ell_2, \ell_2)$. 
    \begin{enumerate}[(A)]
        \item Suppose
    \begin{equation*}
        \norm{T^*}_{2,1} = \sum_i \brac{\sum_j \abs{\ip{T^*\delta_i, \delta_j}}^2 }^{1/2} < \infty,
    \end{equation*}
    i.e., $T^*:c_0 \to \ell_2$ is nuclear in view of Proposition \ref{prop:nuclear-vs-entry}.
    Then $T\in\Pi_2(\ell_2,\ell_1)$ with
    \begin{equation*}
        \pi_2(T)\leq \norm{T^*}_{2,1}.
    \end{equation*}
    \item Suppose $T = T^*$ and $T\in \Pi_2(\ell_2,\ell_1)$. Then $T^*:c_0 \to \ell_2$ is nuclear, with
    \begin{equation*}
        \norm{T^*}_{2,1} = \norm{T}_{2,1}  \leq \kappa_G\cdot \pi_2(T)
    \end{equation*}
    where $\kappa_G$ is the Grothendieck constant in Theorem \ref{thm:Grothendieck}.
    
    \end{enumerate}
    
\end{proposition}

\begin{proof}
    % The boundedness of $T:\ell_2\to \ell_1$ will follow immediately from the $2$-summing estimate. 

    We start with (A). Pick an arbitrary finite collection $x_1, \cdots, x_N \in \ell_2$. For each $k = 1, \cdots, x_N$, we write $x_k = \sum_j \ip{x_k,\delta_j}\delta_j$, so that $Tx_k = \sum_j \ip{x_k,\delta_j}T\delta_j$, and
    \begin{equation}
        \norm{Tx_k}_1^2 = \brac{\sum_i \abs{
        \sum_j \ip{x_k,\delta_j}\ip{T\delta_j,\delta_i}
        }}^2
        = 
        \brac{\sum_i \abs{
        \sum_j \ip{x_k,\delta_j}\ip{\delta_j,T^*\delta_i}
        }}^2
        \label{eq:T21-Txk}
    \end{equation}
    For each $i$, define a unit vector in $\ell_2$ by
    \begin{equation}
        \tau_i:= \frac{T^*\delta_i}{\norm{T^*\delta_i}_2} \text{ if $T^*\delta_i \neq 0$ and $\tau_i:= 0$ otherwise.}
        \label{eq:tau_i}
    \end{equation}
    We then continue the estimate
    \begin{equation}
        \text{\eqref{eq:T21-Txk}} = 
        \brac{\sum_i \norm{T^*\delta_i}_2\abs{\ip{x_k, \tau_i}}}^2
        \leq
        \brac{\sum_i \norm{T^*\delta_i}_2}
        \brac{\sum_i \norm{T^*\delta_i}_2\abs{\ip{x_k,\tau_i}}^2}.
        \label{eq:T21-Txk-2}
    \end{equation}
    Summing in $k$ and using \eqref{eq:T21-Txk-2}, we see that
    \begin{equation*}
        \begin{split}
            \sum_{k=1}^N\norm{Tx_k}_1^2 &\leq 
        \brac{\sum_i \norm{T^*\delta_i}_2}
        \brac{\sum_i \norm{T^*\delta_i}_2
        \sum_{k=1}^N \abs{\ip{x_k,\tau_i}}^2
        }
        \\
        &\leq 
        \brac{\sum_i \norm{T^*\delta_i}_2}^2 \sup_{\norm{\phi}_2 \leq 1}\sum_{k=1}^N \abs{\ip{x_k,\phi}}^2.
        \end{split}
    \end{equation*}
    This concludes the proof of Part (A).

    For Part (B), we use the assumption $T = T^*$ and directly estimate
    \begin{equation*}
        \begin{split}
            \norm{T^*}_{2,1} = \norm{T}_{2,1} = \sum_k\norm{T\delta_k}_2 &\leq \pi_1(T)
            \cdot\sup_{\norm{\phi}_2 \leq 1}\sum_k\abs{\ip{\delta_k,\phi}}
            \\
            &= \pi_1(T)
            \\ 
            &\leq \kappa_G\cdot \pi_2(T)
        \end{split}
    \end{equation*}
    where we use Theorem \ref{thm:Grothendieck} in the last inequality. 
    
\end{proof}

Recall from the spectral theorem that for a positive self-adjoint operator $A\in\L(\ell_2,\ell_2)$, $A^{1/2}\in \Pi_2(\ell_2,\ell_2)$ (i.e., $A^{1/2}$ is Hilbert-Schmidt) if and only if $A$ is trace-class (i.e., $\sum_k\ip{Au_k,u_k} < \infty$ for any orthonormal basis $\set{u_k}$). The following theorem gives the counterpart for operators from $\ell_2 \to \ell_1$ and characterizes a special {\em subfamily} of $\Gamma_+$.

\begin{theorem}\label{thm:square root 2-summing, sufficient, square root diagonal}
    Suppose $A \in \L(\ell_2,\ell_2)$ is positive and self-adjoint, and let
    \begin{equation*}
        D:= \sum_k\sqrt{\ip{A\delta_k,\delta_k}} \in [0,\infty].
    \end{equation*}
    \begin{enumerate}[(A)]
        \item If $D < \infty$, then $A^{1/2} \in \Pi_2(\ell_2,\ell_1)$ with $\pi_2(A^{1/2}) \leq D$. Moreover, by Theorem \ref{thm:2-summing factorization, sufficient Type A}, $A \in \Gamma_+$ with $\gp(A)\leq D^2$.
        \item If $A^{1/2}\in\Pi_2(\ell_2,\ell_1)$ so $A \in \Gamma_+$ by Theorem \ref{thm:2-summing factorization, sufficient Type A}, then $D \leq \kappa_G\cdot \pi_2(A^{1/2})$, with $\kappa_G$ being the Grothendieck constant in Theorem \ref{thm:Grothendieck}.
    \end{enumerate}
\end{theorem}

\begin{proof}
    Note that $A^{1/2}\in \L(\ell_2,\ell_2)$ is self-adjoint. 
    In view of Proposition \ref{prop:T* nuclear implies T 2-summing}, it suffices to show that
    \begin{equation*}
        \norm{(A^{1/2})^*} = \norm{A^{1/2}}_{2,1} = D.
    \end{equation*}
    To see this, we have 
    \begin{equation*}
        \norm{A^{1/2}\delta_k}_2^2 = \ip{A^{1/2}\delta_k, A^{1/2}\delta_k} = \ip{A\delta_k, \delta_k}
        \implies
        \norm{A^{1/2}\delta_k}_2 = \sqrt{\ip{A\delta_k,\delta_k}}.
    \end{equation*}
    Therefore,
    \begin{equation*}
        \norm{A^{1/2}}_{2,1} =  
    \sum_{k}
    \brac{\sum_{l} \abs{\jbrac{A^{1/2}\delta_k,\delta_l}}^{2} }^{\frac{1}{2}}
    =
    \sum_{k}\norm{A^{1/2}\delta_k}_2 = \sum_k \sqrt{\ip{A\delta_k,\delta_k}} = D. 
    \end{equation*}
\end{proof}

We conclude the paper with two remarks.

\begin{remark}
 We clarify that Theorem \ref{thm:square root 2-summing, sufficient, square root diagonal} only concerns the nuclear nature of $A^{1/2}$ and does not say anything about the eigendecomposition of $A$.
Consider the diagonal operator $A\in\L(\ell_2,\ell_2)$ given by $A\delta_k = k^{-2}\delta_k$ for $k \in \mathbb{N}$. $A$ admits an eigendecomposition $A = \sum_k (k^{-1}\delta_k)^*\otimes (k^{-1}\delta_k)$ with $\sum_k \norm{k^{-1}\delta_k}_1^2 = \sum_k k^{-2} < \infty$, so $A\in\Gamma_+$. On the other hand, $A^{1/2}= \sum_k (k^{-1/2}\delta_k)^*\otimes (k^{-1/2}\delta_k)$ and $\sum_k\sqrt{\ip{A\delta_k,\delta_k}} = \sum_k\norm{A^{1/2}\delta_k}_2 = \sum_kk^{-1} = \infty$. By Theorem \ref{thm:square root 2-summing, sufficient, square root diagonal} (or a straightforward verification using $(\delta_k)$), $A^{1/2}\notin \Pi_2(\ell_2,\ell_1)$. 
\end{remark}

\begin{remark}
    There exist positive, self-adjoint, and commuting operators $S,T \in \L(\ell_2,\ell_2)$ such that $S^{1/2}, T^{1/2} \in \Pi_2(\ell_2,\ell_1)$ (so $S,T \in \Gamma_+$ by Theorem \ref{thm:2-summing factorization, sufficient Type A}) but any eigendecomposition of $S+T$ fails to be $\ell_2^\strong(\ell_1)$, and therefore, $(S+T)^{1/2} \notin\Pi_2(\ell_2,\ell_1)$. As in \cite{BOP18}, we define positive self-adjoint operators $S,T \in \L(\ell_2,\ell_2)$ by
\begin{equation*}
    S := \bigoplus_{n=1}^{\infty}\frac{1}{n^{3}}\brac{\sum_{k=1}^n \delta_{\frac{n(n-1)}{2}+k}^*\otimes\delta_{\frac{n(n-1)}{2}+k}}
    \text{\quad and \quad}
    T := \bigoplus_{n=1}^\infty \frac{1}{n^{3+\eps}}\brac{\sum_{k=0}^{n-1}k e_{n,k}^*\otimes e_{n,k}}.
\end{equation*}
In the above $\set{e_{n,k}}_{k=0}^{n-1}$ is the Fourier orthonormal basis of $\C^n$ for each $n \geq 1$, i.e.,
\begin{equation*}
    e_{n,k}(j) = n^{-1/2}\exp\brac{-(2\pi i )\frac{jk}{n}} \text{\quad for $j = 0, 1, \cdots, n-1$.}
\end{equation*}
Note that $S$ and $T$ commute, since $S$ is diagonal. We can immediately read from the expressions of $S$ and $T$ that 
\begin{equation}
    S^{1/2} = \bigoplus_{n=1}^{\infty}\frac{1}{n^{3/2}}\sum_{k=1}^n \delta_{\frac{n(n-1)}{2}+k}^*\otimes\delta_{\frac{n(n-1)}{2}+k}
    \text{\quad and \quad}
    T^{1/2} = \bigoplus_{n=1}^\infty \frac{1}{n^{(3+\eps)/2}}\sum_{k=0}^{n-1}k^{1/2} e_{n,k}^*\otimes e_{n,k}.
    \label{eq:example-1}
\end{equation}
A standard computation verifies that
\begin{equation*}
    \norm{e_{n,k}}_{1} = \sqrt{n}
    \quad \text{ for all $n$ and $k$.}
\end{equation*}
Using the decomposition \eqref{eq:example-1}, we see that
\begin{equation*}
    \gamma(S^{1/2}:\ell_2 \to \ell_1) < \infty
    \text{ and }
    \gamma(T^{1/2}:\ell_2 \to \ell_1) < \infty
\end{equation*}
By \eqref{eq:P&P},
\begin{equation*}
    S^{1/2},T^{1/2} \in \Pi_2(\ell_2,\ell_1).
    \label{eq:example-2}
\end{equation*}
As a consequence, $S,T \in \Gamma_+$ and $S+T \in \Gamma_+$. It is shown in \cite[Proposition 3.1]{BOP18} that $S+T$ only has one-dimensional eigen-subspaces, and an eigendecomposition $\sum_k f_k^*\otimes f_k$ of $S+T$ must satisfy
\begin{equation*}
    \sum_k\norm{f_k}_1^2 = \infty.
    \label{eq:example-3}
\end{equation*}
By Theorem \ref{thm:2-summing factorization, sufficient Type A}, $(S+T)^{1/2} \notin \Pi_2(\ell_2,\ell_1)$.
\end{remark}

\bibliographystyle{plain}
\bibliography{Feichtinger}

@inproceedings{HeilLarson08,
    author = {Christopher Heil and David Larson},
    title = {Operator Theory and Modulation Spaces},
    year = {2006},
    booktitle = {Frames and Operator Theory in Analysis and Signal Processing},
    publisher = {American Mathematical Society},
    series = {Contemporary Mathematics}
}

@article{BOP18,
    author = {Radu Balan and Kasso A. Okoudjou and Anirudha Poria},
    title = {On a problem by {H}ans {F}eichtinger},
    journal = {Operators and Matrices},
    year = {2018},
    volume = {12},
    number = {3}
}

@inproceedings{BORWZ,
    author={Balan, Radu
and Okoudjou, Kasso A.
and Rawson, Michael
and Wang, Yang
and Zhang, Rui},
title={Optimal $\ell^1$ Rank One Matrix Decomposition},
booktitle={Harmonic Analysis and Applications},
year={2021},
publisher={Springer International Publishing},
pages={21--41},
}

@book{DJT95, 
place={Cambridge}, series={Cambridge Studies in Advanced Mathematics}, title={Absolutely Summing Operators}, publisher={Cambridge University Press}, author={Diestel, Joe and Jarchow, Hans and Tonge, Andrew}, year={1995}, collection={Cambridge Studies in Advanced Mathematics}
}

@book{nuclear-book,
   title =     {Traces and Determinants of Linear Operators},
   author =    {Israel Gohberg and Seymour Goldberg and Nahum Krupnik (auth.)},
   publisher = {Birkh\"{a}user},
   series =    {Operator Theory: Advances and Applications   No. 116},
   edition =   {1},
    year = {2000}
}

@article{Sion,
    author = {Maurice Sion},
    title = {On general minimax theorems},
    journal = {Pacific Journals of Mathematics},
    year = {1958},
    volume = {8},
    number = {1},
    pages = {171--176},
}

@article{BMS24,
    title = {A lower bound for the {B}alan–{J}iang matrix problem},
journal = {Applied and Computational Harmonic Analysis},
volume = {73},
pages = {101696},
year = {2024},
author = {Afonso S. Bandeira and Dustin G. Mixon and Stefan Steinerberger},
}

@unpublished{BalanAMS24,
    title = {Primal and dual optimization problems related to matrix factorizations},
    author = {Radu Balan and Fushuai Jiang},
    note = {Special Session on Bases and Frames in Hilbert Spaces, Spring Southeastern AMS Section Meeting},
    year = {March 24, 2024}
}

@article{Anderson1983,
title = {A review of duality theory for linear programming over topological vector spaces},
journal = {Journal of Mathematical Analysis and Applications},
volume = {97},
number = {2},
pages = {380-392},
year = {1983},
author = {Edward~J. Anderson},
}

@incollection{Groth54,
     author = {Alexander Grothendieck},
     title = {Produits tensoriels topologiques et espaces nucl\'eaires},
     booktitle = {S\'eminaire Bourbaki : ann\'ees 1951/52 - 1952/53 - 1953/54, expos\'es 50-100},
     series = {S\'eminaire Bourbaki},
     note = {talk:69},
     pages = {193--200},
     publisher = {Soci\'et\'e math\'ematique de France},
     number = {2},
     year = {1954},
     mrnumber = {1609222},
     language = {fr},
}

@book{BV-convexbook,
  title={Convex Optimization},
  author={Stephen Boyd and Lieven Vandenberghe},
  year={2004},
  publisher={Cambridge University Press}
}

@inproceedings{HallNewman1963-completelypositive,
  title={Copositive and completely positive quadratic forms},
  author={Hall, Marshall and Newman, Morris},
  booktitle={Mathematical Proceedings of the Cambridge Philosophical Society},
  volume={59},
  number={2},
  pages={329--339},
  year={1963},
  organization={Cambridge University Press}
}

@article{friedland2020-Grothsym,
  title={Symmetric grothendieck inequality},
  author={Friedland, Shmuel and Lim, Lek-Heng},
  journal={arXiv preprint arXiv:2003.07345},
  year={2020}
}

@article{pisier2012grothendieck,
  title={Grothendieck’s theorem, past and present},
  author={Pisier, Gilles},
  journal={Bulletin of the American Mathematical Society},
  volume={49},
  number={2},
  pages={237--323},
  year={2012}
}

@article{paneque2000banach,
  title={On {B}anach spaces {$X$} for which {$\Pi_2(X,\mathcal{L}_1) = \mathcal{N}_1(X,\mathcal{L}_1)$}},
  author={Paneque, F and Pi\~{n}eiro, C},
  journal={Quaestiones Mathematicae},
  volume={23},
  number={3},
  pages={287--294},
  year={2000},
  publisher={Taylor \& Francis}
}

\end{document}